\documentclass[10pt,a4paper,twoside]{article}

\makeatletter
\let\@fnsymbol\@arabic
\makeatother

\usepackage{amsmath,amsfonts,amssymb,latexsym,wasysym,mathrsfs}
\usepackage{amsthm}
\usepackage{verbatim}
\usepackage[english]{babel}
\usepackage{a4wide}
\usepackage{pstricks,graphicx,subfigure}

\usepackage[footnotesize,bf,centerlast]{caption}
\usepackage[small,euler-digits,icomma,OT1,T1]{eulervm}
\usepackage[all,2cell]{xy} \UseAllTwocells \SilentMatrices
\usepackage{todonotes}
\usepackage[normalem]{ulem}
\usepackage{enumitem}

\theoremstyle{plain}
\newtheorem{theorem}{Theorem}[section]
\newtheorem{proposition}{Proposition}[section]
\newtheorem{lemma}{Lemma}[section]

\theoremstyle{definition}

\theoremstyle{remark}
\newtheorem{remark}{Remark}[section]
\newtheorem*{remark*}{Remark}

\def \R {{\mathbb R}}

\def \ra {\rightarrow }

\def \o {\omega}

\def \s {\sigma}

\def \e {\epsilon}

\def \L {\Lambda}
\def \l {\lambda}
\def \b {\beta}
\def \g {\gamma}
\def \a {\alpha}

\def \ra {\rightarrow}
\def \uN {\Lambda_N}
\title{Rhythmic behavior of an Ising Model\\ with dissipation at low temperature}

\author{Rapha\"el Cerf 
\footnote{
DMA, Ecole Normale Sup\'erieure, CNRS, PSL University, 75005 Paris.
\newline\vskip-13pt
\indent\kern7pt Laboratoire de Math\'ematiques d'Orsay, CNRS, Universit\'e Paris– Saclay, 91405 Orsay.
} 
\and Paolo Dai Pra 
\footnote{Department of Mathematics ``Tullio Levi-Civita'', University of Padova, Via Trieste 63, 35121 Padova, Italy.}\, 
\footnote{Department of Computer Science, University of Verona, Strada Le Grazie 15, 37134 Verona, Italy.
}
\and Marco Formentin
\footnotemark[2]\,
\footnote{Padova Neuroscience Center, University of Padova, via Giuseppe Orus 2, 35131 Padova, Italy.
}
\and Daniele Tovazzi \footnotemark[2]
}

\date{}

\begin{document}

\maketitle

\begin{abstract}
\noindent 
In this paper we consider the Glauber dynamics for the one-dimensional Ising model with dissipation, in a mesoscopic regime obtained by letting inverse temperature and volume go to infinity with a suitable scaling. In this limit the magnetization has a periodic behavior. Self-organized collective periodicity has been shown for many mean-field models but, to our knowledge, this is the first example with short-range interaction. This supports the view that self-organized periodicity is not linked with the mean-field assumption but it is a thermodynamic phenomenon compatible with short range interactions.
\vspace{0.3cm}

\noindent {\bf Keywords:} Self-organized complex  systems $\cdot$ Ising model with dissipation $\cdot$ Collective rhythmic behavior \\ \\
\end{abstract}
%{\bf 2010 MSC Classification:}

%====================================================================
\section{Introduction}
%====================================================================
Rhythmic behavior emerges in many biological and socioeconomic complex systems \cite{Iz07,turchin,weidlich}, and may involve a wide range of time scales: from the fraction of a second of neural rhythms, to the years of ecological and epidemiological rhythms. Such a behavior cannot be ascribed to the single units of the system (e.g. cells or individual animals) but it is the result of the interactions within the network. In recent years many stylized models have been proposed to identify possible origines of time-periodicity \cite{GiPo15,LiG-ONeS-G04,Sch85a,Sch85b,ShKu86,AlMi}. Existing examples are mostly restricted to mean-field interaction, i.e., the interaction network is the complete graph \cite{AnTo18,DaPFiRe13,CoDaPFo15,CoFoTo16,JK14,LuPo_b,LuPo_a}. It has been shown that periodicity may emerge in the thermodynamic limit in presence of some time-symmetry breaking features, such as {\em dissipation} \cite{CoDaPFo15,DaPFiRe13}, {\em delay} \cite{DiLo17,touboul19}, {\em asymmetry in the pair interaction} \cite{CoFoTo16,FeFoNe09}. %{\r In \cite{JK14} authors show the existence of a locally attractive periodic orbit of rotating measures for discrete mean-field rotators.}

In this paper we consider a dissipative version of the Glauber Dynamics on the Ising model with nearest neighbor interaction. The dissipative model is obtained from the standard Glauber Dynamics by introducing a linear mean-reversion that drives the logarithm of the rates to a reference value (e.g. zero) in the intervals between two consecutive spin-flips. The corresponding mean-field model has been fully solved in \cite{DaPFiRe13}. The picture that emerged is the following. The Glauber dynamics, as the number $N$ of spins diverges to infinity, converges to a deterministic limit (macroscopic) evolution. In absence of dissipation this evolution can be expressed in term of a single scalar parameter, the magnetization, which evolves according to a nonlinear ordinary differential equation. At the critical value of the inverse temperature $\beta_c = 1$ this evolution exhibits a {\em pitchfork bifurcation}: for $\beta \leq \beta_c$ the equilibrium $m=0$ is a global attractor, while as $\beta > \beta_c$ two nonzero stable equilibria bifurcate from the null solution. As dissipation is turned on, the  macroscopic evolution can be reduced to a two dimensional ordinary differential equation, still possessing a critical value $\beta_c$ for the inverse temperature, which now becomes a {\em Hopf bifurcation}: as $\beta > \beta_c$ a unique stable periodic orbit stems from the null solution.

Our aim is to show that the macroscopic evolution of the magnetization may be time-periodic also in the dissipative nearest neighbor Ising model. The mean-field case suggests that periodic orbits emerge in the dissipative model for temperatures that would lead to spontateous magnetization in the non dissipative system. It would therefore be natural to consider the nearest neighbor Ising model on the two dimensional lattice. This is however far beyond our mathematical understanding. Note that, unlike in the mean-field case, there is no way of reducing the dynamics in the thermodynamic limit to a finite-dimensional dynamical system. Thus  we base our analysis on the asymptotic dynamics of droplets. Unfortunately, we do not have sufficient control on the dynamics of the droplets in the two dimensional case. For this reason we consider the dissipative Ising model in one dimension. It is well known that, at any fixed positive temperature, no spontaneous magnetization occurs. However magnetization can be produced by letting the inverse temperature $\beta$ diverge as the volume $N$ goes to infinity. An elementary computation based on the transfer matrix with positive boundary condition shows that whenever $N = o\left(e^{2\beta}\right)$ the limit magnetization in equilibrium equals one, while magnetizations between zero and one are obtained when $N$ is of order $e^{2\beta}$. 

In this paper we assume the stronger condition 
\[
\frac{\ln N}{\beta} \rightarrow c \in [0,1[.
\]
Note that the condition $c<2$ would suffice for the equilibrium magnetization to be one. However the microscopic dynamics of the dissipative system changes dramatically as $c$ crosses $1$. For $c<1$, starting from all equal spins, a droplet of opposite sign forms and invades the space with high probability before the formation of other droplets. For $1<c<2$, several droplets form and merge before the space is invaded. This last situation is more complicated and requires further work, not yet fully under control. Our analysis is based on the study of the distribution of two stopping times: $T_1$ is the time the first spin flip occurs, i.e., when the first droplet forms; $T_c$ is the time needed after $T_1$ for the initial configuration of, say, all negative spins, to be replaced by all positive spins (covering time). We prove, under a condition weaker than the $c<1$ we just mentioned, that $T_1$, when properly rescaled, has a deterministic limit, with also a control on the fluctuations. Note that this differs from the usual Glauber dynamics with no dissipation, where $T_1$ is simply exponential. When $c<1$, after the occurrence of the first spin flip, with overwhelming probability, the droplet grows with linear speed until it fills the space. This covering time is much smaller than $T_1$, so at the time scale of $T_1$ we observe periodic pulsing between homogeneous configurations.

%We conclude by mentioning that cyclic temporal patterns can be produced also in a mean-field game-theoretic setting \cite{DaPSaTo13, DaPSaTo}.\\

The paper is organized as follows. In Section~\ref{sect:model_and_results} we describe the model under consideration and state our main results. All the proofs are postponed to Section~\ref{sect:proofs}.

\bigskip

%====================================================================
%\section{}
%===================================================================
%%%%%%%%%%%%%%%%%%%%
\section{Description of the model and results}
\label{sect:model_and_results}
%%%%%%%%%%%%%%%%%%%
In this section we present an Ising model with dissipation and we describe the results we aim to prove.\\

Let $\mathcal{S}=\{-1,+1\}$ and consider a configuration of $N$-spins $\underline{\s}\in \mathcal{S}^{\Lambda_N}$, where $$\Lambda_N=\{1,2,\dots,N\}\subseteq \mathbb{Z}$$ represents the set of sites of the spins. We assume periodic boundary condition, i.e. $\sigma_{N}\equiv \sigma_{1}$ and $\sigma_0 \equiv \s_N$.\\
The stochastic Ising model with dissipation $\a \geq 0$ and inverse temperature $\b>0$ is the Markov process 
$(\underline{\s}(t),\underline{\l}(t))_{t\geq 0}$ with values in $\mathcal{S}^{\Lambda_N}\times \R^N$ evolving according to the following dissipated dynamics: at a given time $t\geq 0$, each transition $\sigma_i(t)\to-\sigma_i(t)$, $i\in \Lambda_{N}$, 
occurs with rate 
\begin{equation}\label{rate}
r_i(t):=\exp(-\sigma_i (t) \lambda_i(t)),
\end{equation}
 where $\{\lambda_i(t)\}_{i\in\Lambda_{N}}$ is a family of stochastic processes ({\em local fields}) evolving according to 
\begin{equation}
\label{lambda} d\lambda_i(t)= -\alpha \lambda_i(t) dt + \beta d m_i(t), \:\:\:\:\: i\in\Lambda_{N}
\end{equation}
with $\alpha,\b>0$ and \begin{equation}
\label{magnetization} m_i(t)= \sum_{j\sim i} \sigma_j(t), \:\:\:\:\: i\in\Lambda_{N},
\end{equation}
where $j\sim i$ denotes the set of sites $j$ which are neighbors of $i$ (namely, $i-1$ and $i+1$). 
Formally speaking, $(\underline{\s}(t),\underline{\l}(t))_{t\geq 0}$ is a Markov process with infinitesimal generator
\begin{equation}
\label{ch_three_infgenIsing} \mathcal{L}_Nf(\underline{\s},\underline{\l})=\sum_{i\in \Lambda_N}\exp[-\s_i\l_i]\left(f(\underline{\s}^i,\underline{\l}-2\b \s_i \underline{v}^i)-f(\underline{\s},\underline{\l})\right)-\a\l_i f_{\l_i}(\underline{\s},\underline{\l}),
\end{equation}
where $f_{\l_i}$ represents the partial derivative of $f$ with respect to $\l_i$, $\underline{\sigma}^i$ is the configuration obtained by flipping the state of the $i$-th spin and $\underline{v}^i$ is a $N$-dimensional vector such that
\[
v^i_k=\begin{cases}
1, \:\:\:\: &k=i+1 \mbox{ or } k=i-1,
\\
0, \:\:\:\: &\mbox{otherwise}.
\end{cases}\]
In what follows, we will assume
initial conditions on the form
:
\begin{equation}
\label{ch_three_initialconditions} \s_i(0)=-1, \:\:\:\: \l_i(0)=-\l_{N,\b}(i), \:\:\:\: \mbox{ for any } i\in\Lambda_N.
\end{equation}
\begin{remark}
By taking $\a=0$ (i.e. ruling out dissipation), we obtain a Glauber dynamics for the classical 1-dimensional Ising model with periodic boundary conditions, inverse temperature $\b$ and magnetic fields $\l_i(0)$.
\end{remark}

Our aim is to show that in a suitable large volume - low temperature limit, the total magnetization of the system has a rhythmic behavior after a proper time scaling: we briefly describe the phenomenon here. \\
Assuming initial conditions \eqref{ch_three_initialconditions}, the analysis of the evolution of $(\underline{\s},\underline{\l})_{t\geq 0}$ is divided into two parts. We begin by studying the occurrence time of the first spin flip. Unlike the case with no dissipation ($\a = 0$), where this time is exponentially distributed, the dissipation produces a much higher concentration of the distribution of this time: indeed, it will converge to a deterministic value as $\gamma,N\uparrow+\infty$. After the first spin-flip occurs, the change in the local field and the low temperature ($\b\uparrow+\infty$) favours the growth of a ``droplet'' (just a segment in the one-dimensional case) of $+1$ spins, which invades the whole state space in an extremely short time scale. At this point we are back to the situation of all equal spins. We will show that by assigning the initial local fields $\l_i(0)$ in a suitable way, the local fields at the time the droplet has invaded the space is essentially opposite to the initial one, producing the iteration of the same phenomenon. Since the two parts of the evolution (waiting for the first spin-flip and covering by the droplet) occur on different time scales, we will consider a time-rescaled magnetization process to analyse the macroscopic behavior.\\
To guarantee that the phenomenon described above occurs with overwhelming probability, we will assume that $\b,N \uparrow \infty$ in such a way that ${\ln N\over \beta }\to c\in[0,1[$. This assumption guarantees that, after the first spin-flip, the droplet of $+1$ spins covers the whole space \textit{before} the birth of other droplets. As we will see in Section \ref{section:covertime}, this allows a good understanding for the time taken by the droplet to cover $\Lambda_N$. Indeed, if ${\ln N\over \beta} \to c \in[1,2[$, a single droplet cannot invade the whole space: 
in this case, the box of size $N$ is too big to be covered by a single droplet, and many other droplets of $+1$ spins appear. We believe that this does not rule out periodic behavior, but it makes the analysis considerably harder. \\

In what follows, we will see that in the regime ${\ln N\over \beta }\to c\in[0,1[$ the waiting time for the first spin flip is large, but has small fluctuations. These fluctuations, however, have an impact on the growth time of the droplet. For this reason, while the waiting time of the first spin flip, rescaled by its mean, has a deterministic limit, the rescaled growth time of the droplet keeps some randomness in the limit. Due to this fact, the macroscopic evolution will not be 
strictly
periodic, but it will present regular oscillations with stochastic rhythm.\\

%In Section \ref{section:firstspinflip} we analyse the limiting distribution of the stopping time at which the first spin flip is observed, in Section \ref{section:covertime} we study the time scale at which the covering occurs. Finally, these results will be exploited in Section \ref{oscillation}, where the limiting behavior of the magnetization is presented.

Before stating our main results, we introduce the {\em graphical construction} of the process, that will be useful in the proofs to  couple it with other processes.

Let $\{\mathcal{N}_i\}_{i\in\mathbb{N}}$ be a family of i.i.d. Poisson processes of intensity $e^{4\b}$ and denote the successive arrival times of the $i$-th Poisson process with $\{\tau_{i,n}\}_n$. Each arrival time $\tau_{i,n}$ is associated with a random variable $U_{i,n}$, uniformly distributed on $[0,1]$. The random variables $\{U_{i,n}\}_{i,n}$ are independent among themselves and independent from the Poisson processes $\{\mathcal{N}_i\}_{i\in\mathbb{N}}$. This concludes the construction of the probability space. For a fixed $N>1$, the process $(\underline{\sigma},\underline{\lambda})$ evolves as follows: each site $i\in\Lambda_N$ is associated with the process $\mathcal{N}_i$; then, each point $\tau_{i,n}$ is \textit{accepted} for a spin flip only if 
\[
{\exp[-\sigma_i(\tau_{i,n})\lambda_i(\tau_{i,n})]\over e^{4\b}}
>
U_{i,n}.
\]
Whenever a point $(\tau_{i,n})$ is accepted, the spin at site $i$ is flipped and the values of the local fields are updated in the following way:
\[ \lambda_k(\tau_{i,n})=\begin{cases}
\lambda_k(\tau_{i,n}^-)-\sigma_i(\tau_{i,n}^-)2\beta, \:\:\:\:\:\:\:k=i+1,i-1\\
\lambda_k(\tau_{i,n}^-), \:\:\:\:\:\:\:\:\:\:\:\:\:\:\:\:\:\:\:\:\:\:\:\:\:\:\:\:\:\:\:\:\:\mbox{otherwise}
\end{cases}
\]
At any time in which there is no accepted spin flips, the local fields evolves according to 
\[
 \dot{\lambda}_i(t)= -\alpha \lambda_i(t), \:\:\:\:\: i\in\Lambda_{N}.
\]
One can check that this construction provides the rates prescribed by \eqref{ch_three_infgenIsing}.
Other processes will be later coupled with $(\underline{\sigma},\underline{\lambda})$ using this graphical construction.

\subsection{First spin flip}

In this paper we will prove asymptotic results in the limit as $\beta$ and $N$ go simultaneously to infinity. In this section we assume the {\em low temperature} condition
\begin{equation} \label{betalarge}
\lim \frac{Ne^{-\b}}{\b} = 0,
\end{equation}
that is weaker that what we will assume later on. Here and later, ``$\lim$'' stands for ${\displaystyle{\lim_{N,\beta \uparrow +\infty}}}$.

We begin by considering initial conditions of the form
\begin{equation} \label{inhom_lambdas}
\sigma_i(0)=-1\,,\qquad
	\lambda_i(0)\,=\,\lambda_{N,\beta}(i) \,= \, \l_{N,\b} + o(i,N,\b)
	\,,
	\qquad
1\leq i\leq N\, ,
\end{equation}
where $\l_{N,\b}$ is a family of real numbers, $o(i,N,\b)$ is a family of real random variables such that, as $N$ and $\beta$ go simultaneously to infinity, the following condition holds:
\begin{equation} \label{nearlyc}
	\forall\e>0\qquad
\lim P\left( \sup_i |o(i,N,\b)| > \e\right) = 0\,.
\end{equation}
This last condition states that the initial local fields are {\em nearly constant}. More general initial conditions for the local fields will be considered later. To avoid unnecessary complications, we also assume that the limit
\[
\lim \frac{\l_{N,\b}}{\ln N}
\]
exists.

Our first goal is to study the time $T_1$ of the first spin flip, defined as
$$T_1\,=\,\inf\,\big\{\,t\geq 0:\exists\,i\in \uN,\:\sigma_i(t)=1\,\big\}\,.$$

\begin{theorem} \label{th:firstspinflip}
Under assumptions \eqref{betalarge}, \eqref {inhom_lambdas} and \eqref{nearlyc}, we have the following asymptotic behavior:
\begin{itemize}
\item[(a)] 
if $\lim \frac{\l_{N,\b}}{\ln N} < -1$ then
\begin{equation} \label{xNbeta}
X_{N,\beta}:= \a \ln N \left( T_1 - \frac{1}{\a} \ln \left( \frac{- \l_{N,\beta}}{\ln N}\right)- \frac{\ln \ln N}{\a \ln N} - \frac{ \ln \a}{\a \ln N}\right)
\
\end{equation}
converges in distribution as $\b,N \ra +\infty$ to a random variable $X$ whose distribution is given by
\begin{equation} \label{X}
P(X>x) = \exp\left(- e^x \right);
\end{equation}
\item[(b)]
if $\lim \frac{\l_{N,\b}}{\ln N} \geq -1$ then $T_1$ converges to zero in probability as $\b,N \ra +\infty$.
\end{itemize}

\end{theorem}
Theorem \ref{th:firstspinflip} states that, provided $-\l_{N,\b}$ is sufficiently large, 
\[
T_1 = \frac{1}{\a} \ln \left( \frac{- \l_{N,\beta}}{\ln N}\right) +  \frac{\ln \ln N}{\a \ln N}  + \frac{\ln \a}{ \a \ln N} + \frac{X_{N,\beta}}{\a \ln N} + o\left(\frac{1}{\ln N} \right).
\]
We will later choose $\l_{N,\beta}$ in such a way that $\ln \left( \frac{- \l_{N,\beta}}{\ln N}\right)$ converges to a strictly positive constant.

To analyze
the evolution of the system after $T_1$, we will also compute the value of the local fields immediately before the first spin flip:
\begin{equation} \label{lambdaT1bis}
\l_i(T_1^-) = \l_i(0)e^{-\a T_1} = - \ln N + \ln \ln N + \ln \a + X_{N,\b} + o(i,N,\b),
\end{equation}
where $o(i,N,\b)$ satisfies \eqref{nearlyc}. Thus the initial value $\l_{N,\beta}$ is essentially ``forgotten'' at time $T_1$.

\subsection{Covering time} \label{sec:covtime}

Now we study the evolution of the spin system after time $T_1$, so consider the processes
$(\tilde{\underline{\s}}(t),\tilde{\underline{\l}}(t))_{t\geq 0}$ such that $$ \begin{array}{c} \tilde{\sigma}_i(t)=\sigma_i(t+T_1)\\ \tilde{\lambda}_i(t)=\lambda_i(t+T_1)
\end{array} \:\:\:\:\:\:\:\:\:\:\:\: t\geq 0,\:\: i\in\Lambda_{N}.$$
By the strong Markov property, the evolution of $(\tilde{\underline{\s}}(t),\tilde{\underline{\l}}(t))_{t\geq 0}$ is still described by \eqref{rate} and \eqref{lambda}. Define
\[
T_c := \inf\{t > 0 : \tilde{\sigma}_i(t) = 1 \ \mbox{ for all } i\in\Lambda_N\}
\]
the time needed to reach the homogeneous configuration with all spins equal to $+1$. The following theorem describes the asymptotic behavior of $T_c$ as $\b,N \uparrow +\infty$, and it implies, in particular, that $T_c \rightarrow 0$ as $N \ra +\infty$. In what follows we assume \eqref{inhom_lambdas} as initial condition for the local fields.

\begin{theorem} \label{th:covering}
Let $\b,N \uparrow +\infty$ in such a way that
\begin{equation}\label{condbeta}
\lim \frac{\ln N}{\b} = c\,, \ \ \ \lim \frac {\l_{N,\b}}{\ln N} < -1\,,
\end{equation}
with $c\in[0,1[$,
and assume that conditions \eqref{inhom_lambdas} and \eqref{nearlyc} hold.
Then
\begin{equation}\label{randscal}
\frac{T_c}{\frac{N^2}{2\a \ln N}e^{-2\b- X_{N,\beta}}} \ \ \overset{P}{{\longrightarrow}} \ \ 1
\end{equation}
in probability, and
\begin{equation}\label{detscal}
\frac{T_c}{\frac{N^2}{2\a \ln N}e^{-2\b}} \ \ \overset{P}{{\longrightarrow}} \ \ Z,
\end{equation}
where $X_{N,\b}$ is defined in \eqref{xNbeta} and $Z$ is a random variable distributed as $e^{-X}$, with $X$ being the random variable introduced in \eqref{X}. Moreover
\begin{equation} \label{newfield}
\l_i(T_1+T_c) = 4\b - \ln N + \ln \ln N + \a + X_{N,\b} + o(i,N,\b),
\end{equation}
where $o(i,N,\b)$ satisfies \eqref{nearlyc}.
\end{theorem}

\begin{remark} \label{rem:droplet}
The proof of Theorem \ref{th:covering} is based on the fact that, with probability that goes to one as $N \ra +\infty$, the droplet of spin $+1$ that forms at time $T_1$ grows at nearly constant speed up to the covering time $T_c$. This fact holds true, with no change in the proof, even if 
\[
\lim \frac {\l_{N,\b}}{\ln N} \geq  -1.
\]
In this case $\l_i(T_1^-) \simeq \l_{N,\b}$ as $T_1 \simeq 0$, so
\[
\l_i(T_1+T_c) = 4\b + \lambda_{N,\beta} + o(i,N,\b).
\]
Thus, at time $T_1+T_c$ the state of the system is the same as the initial state~\eqref{inhom_lambdas}
	with all signs changed: all spins equal $+1$ and the local fields are nearly constant. Moreover 
\[
\lim \frac{\l_i(T_1+T_c)}{\ln N} >1,
\]
due to the condition $\lim \frac{\ln N}{\b} = c \in [0,1]$. This allows to iterate the analysis.

\end{remark}

\subsection{Oscillating behavior}

The result of Section \ref{sec:covtime} shows that starting with an initial condition which is constant (say $-1$) for the spins  and nearly constant for the local fields, after two droplet expansions, with probability that goes to one as $N \ra +\infty$, the systems reaches a state of the form \eqref{inhom_lambdas}, with 
\begin{equation} \label{newlambda}
\l_i = -4 \b + \ln N - \ln \ln N + O(i,N)
\end{equation}
where $O(i,N)$ is a bounded correction. It is therefore natural to assume $\l_i(0)$ to be as in \eqref{newlambda}. By the results above the following facts follow.

\begin{theorem}\label{thm:many_stopping}
Let $ \gamma_{N,\b}=4\beta -\ln N +\ln\ln N$ and take the initial conditions
\[ \sigma_i(0)=-1,\:\:\:\:\:\lambda_i(0)=-\gamma_{N,\b} + O(i,N),\:\:\:\:\:\:\:\:\:\: i\in\Lambda_N.\]
%as in \eqref{newlambda}.
Fix $n\in\mathbb{N}$ and define the following stopping times, for $j=1,\dots,n$
\[ 
\begin{split}
T_{1,j}&:=\inf\left\{t>\sum_{k=0}^{j-1}(T_{1,k}+T_{c,k})\:\Big|\: \sigma_i(t)=(-1)^{j+1} \:\: \mbox{\em{for some}}\:\: i\in\Lambda_N\right\}-\sum_{k=0}^{j-1}(T_{1,k}+T_{c,k}),\\
T_{c,j}&:=\inf\left\{  t > T_{1,j}+\sum_{k=0}^{j-1}(T_{1,k}+T_{c,k}) \:\Big|\: {\sigma}_i(t) = (-1)^{j+1} \:\: \mbox{\em{for all}}\:\: i\in\Lambda_N\right\}-\sum_{k=0}^{j-1}(T_{1,k}+T_{c,k})
\end{split}
\]
with $T_{1,0}=T_{c,0}=0$. Let $\{Y_i\}_{i=1}^n$ be a sequence of i.i.d. random variables distributed according to 
$$P(Y_1>y)=\exp\left(-e^y\right), \:\:\:\:\:\:\:\:\forall\:\: y \in \mathbb{R}.$$
Suppose $\beta,N\uparrow +\infty$ with the condition \[ \lim_{\beta,N}{\ln N \over \beta}=c\in[0,1[.\]
Then, for any $j=1,\dots,n$, 
\[
\begin{split}
\alpha \ln N \left(T_{1,j}- \frac{1}{\a} \ln \left( \frac{- \g_{N,\b}}{\ln N}\right)- \frac{\ln \ln N}{\a \ln N} - \frac{ \ln \a}{\a \ln N}\right) \overset{d}{\underset{\gamma,N\uparrow+\infty}{\longrightarrow}} Y_j\\
 {T_{c,j}\over {N^2\over 2\alpha \ln N} e^{-2\beta}} \overset{d}{\underset{\b,N\uparrow+\infty}{\longrightarrow}} Z_j,
\end{split}
\]
where $Z_j$ is distributed as $e^{-Y_j}$.
\end{theorem}

These results show that the system is governed by
two time scales. The first spin flip is concentrated around the time
\[
t(1,N) = \frac{1}{\a} \ln \left( \frac{- \g_{N,\b}}{\ln N}\right)- \frac{\ln \ln N}{\a \ln N} - \frac{ \ln \a}{\a \ln N} \sim \frac{1}{\a} \ln \left(\frac{4-c}{c} \right) 
\]
as $N \ra +\infty$, $\frac{\ln N}{\b} \ra c \in [0,1[$ (note that if $c=0$ then $t_N \ra +\infty$). The droplet expansion occurs at the time scale
\[
t(c,N) = \frac{N^2 e^{-2\b}}{2\a \ln N},
\]
which goes to zero as $N \ra +\infty$,{$\frac{\ln N}{\b} \ra c \in [0,1[$}. This suggests to consider the time change for the magnetization process
\[
m_N(t) := \frac{1}{N} \sum_{i \in \L_N} \s_i(t)
\]
given by
\begin{equation} \label{timescaling}
\theta_N(t) = t(1,N) \int_0^t {\bf 1}_{\{|m_N(t)| = 1\}}dt +  t(c,N) \int_0^t {\bf 1}_{\{|m_N(t)| < 1\}}dt,
\end{equation}

which ``speeds up" time whenever all the spins are equal and we are waiting for the following flip and ``slows down" time whenever we are observing the very fast invasion of a droplet of spins of the opposite sign. 
Then, we define a time-scaled version of the total magnetization process by
\begin{equation}\label{tilde_m}
\tilde{m}_N(t):= m_N(\theta_N(t)).
\end{equation} 
By Theorem \ref{thm:many_stopping} and the analysis performed in the proof of Theorem \ref{th:covering}, we expect that the process $\tilde{m}_N$ converges to a stochastic process $\tilde{x}$ with the following behavior: $\tilde{x}(0)=-1$, then it does not move for a unit of time, then it takes a random time $Z_1$ to linearly grow from $-1$ to $+1$; after reaching $+1$, it does not move for a unit of time, then it takes a random time $Z_2$ to linearly decrease from $+1$ to $-1$ and so on, where the random variables $Z_1,Z_2,\dots$ are given in Theorem \ref{thm:many_stopping}. We expect a linear profile between $-1$ and $+1$ and also between $+1$ and $-1$
since in the proof of Theorem \ref{th:covering} we saw that during the growth of the droplet each step occurs essentially at the same time. 
%A graphical intuition on the behavior of the limiting dynamics is given in Figure \ref{ch_three_pictureIsing}.

Let us give a formal definition of the limiting process $\tilde{x}$: consider the deterministic trajectory $x(t)$ such that
\[  x(t)=\begin{cases}
-1  \:\:\: &\mbox{ for }t\in [0,1[,\\
2t-3  \: &\mbox{ for }t\in [1,2[,\\
+1 \: &\mbox{ for }t\in [2,3[,\\
-2t+7 \: &\mbox{ for }t\in [3,4[,\\
\end{cases}\]
and then extended periodically on $\mathbb{R_+}$ for $t\geq 4$. Then, consider the family of random variables $\{Z_i\}_{i\geq 1}$ defined in Theorem \ref{thm:many_stopping} and define the following time-changing process:\begin{equation}
	\label{phi_x} \phi(t)=\int_0^t\Big( {\bf 1}_{\{|x(s)|=1\}} + \sum_{i\geq 1} Z_i^{-1} {\bf 1}_{\{ s\in [2i-1,2i[\}}\Big) ds\,.  
\end{equation}
Finally, the limiting process is defined as 
\begin{equation}
\label{tilde_x}
\tilde{x}(t)=x(\phi(t)).
\end{equation}
\begin{theorem}\label{thm:macroscopic_oscillations}
Let $\gamma$ and $\{Z_i\}_{i\geq 1}$ as in Theorem \ref{thm:many_stopping}. Suppose $\b,N\uparrow+\infty$ with the condition ${\ln N\over \b}=c\in[0,1[$. Then, for any $T>0$, the process $(\tilde{m}_N(t))_{t\in[0,T]}$ defined by \eqref{tilde_m} converges, in the
	sense of weak convergence of stochastic processes, to $(\tilde{x}(t))_{t\in[0,T]}$ defined by \eqref{phi_x}-\eqref{tilde_x}.
\end{theorem}

\subsection{Smoothly varying initial condition.}

We have seen that if the initial local fields are nearly constant then the system evolves by periodic droplet formations and expansions; moreover 
when the droplet expansion terminates, the local fields are nearly constant with absolute value given by 
\[
4 \b - \ln N + \ln \ln N + \a + X_i,
\]
up to corrections that vanish in the limit as $N \ra +\infty$, where the $X_i$'s are random variables, independent for different iterations, with distribution
\[
P(X_i > x) = \exp\left(-e^{x}\right).
\]
We show next that these nearly constant profiles for the local fields are stable under perturbations that are sufficiently small and regular.
More specifically, assume 
\begin{equation} \label{varying}
\l_i(0) = \l_{N,\b} \Phi\left(\frac{i}{N}\right) + o(i,N,\b),
\end{equation}
where $o(i,N,\b)$ is as in \eqref{nearlyc}, $\Phi:[0,1] \ra (0,+\infty)$ is a $\mathcal{C}^2$ function, with a unique minimum $x_* \in (0,1)$, and $\Phi''(x_*)>0$ and a unique maximum $x^*$, with $\Phi''(x^*) <0$. With no loss of generality, 
%possibly rescaling $\l_{N,\b}$ 
we assume that $\Phi(x_*) = 1$. In the following theorem we also assume the usual initial condition $\s_i(0) \equiv -1$.

\begin{theorem} \label{th:varying}
Assume that $\Phi$ takes its values in $[1,2]$, $\lim \frac{\l_{N,\b}}{\ln N}\leq -1$ and $\lim\frac{\ln N}{\b} = c \in [0,1[$. 
\begin{itemize}
\item[(i)]
Denoting by $T_{1,1}$ the time of the first spin flip,
\begin{equation} \label{T1var}
\a \ln N \left( T_{1,1} 
- \frac{1}{\a} \ln \left( - \frac{\l_{N,\b}}{\ln N} \right) 
- \frac{3}{2\a} \frac{\ln \ln N}{\ln N} 
- \frac{1}{\a \ln N} \ln \left( \a \sqrt{\frac{\Phi''(x_*)}{2 \pi}} \right) 
\right)
\end{equation}
converges in distribution to a random variable $X$ as in \eqref{X}.

\item[(ii)]
Using the notations introduced in Theorem \ref{thm:many_stopping}, let $T_{1,1} + T_{c,1}$ be the first time all spins become $-1$. Then $T_{c,1}$ converges to zero as $N \ra +\infty$ in probability, and
\[
\l_i(T_{1,1} + T_{c,1}) = 4 \b - \Phi\left(\frac{i}{N}\right) \left[ \ln N - \frac32 \ln \ln N - \ln \left( \a \sqrt{\frac{\Phi''(x_*)}{2 \pi \Phi(x_*)}} \right) - X_{N,1} \right] + o(i,N,\b),
\]
where $X_{N,1}$ converges in distribution to a random variable $X$ as in \eqref{X}.

\item[(iii)]
After the $j$-th droplet expansion, $j \geq 2$, the local fields are given, up to the sign, by
\begin{multline*}
\left|\l_i\left(\sum_{k=1}^j (T_{1,k} + T_{c,k}) \right)\right| \\ =
4 \b - \left(R^{j-1}\Phi\right)\left(\frac{i}{N}\right) \left[ \ln N - \frac32 \ln \ln N - \ln \left( \a \sqrt{\frac{(R^{j-1}\Phi)''(x^*)}{2 \pi (R^{k-1}\Phi)(x^*)}} \right) - X_{N,j} \right] + o(i,N,\b), 
\end{multline*}
where the sequence $(X_{N,j})_{j \geq 1}$ converges to an i.i.d sequence, and $R^{j-1}$ denotes the $j-1$-st iteration of the map
\[
R \phi (x) = \frac{4 -c \Phi(x)}{4-c \Phi(x^*)}.
\]

\end{itemize}

\end{theorem}

Noting that for each $x \in [0,1]$ we have $R^j \Phi(x) \ra 1$ as $j \ra +\infty$, this last result shows that nearly constant profiles are attracting.

\section{Proofs} \label{sect:proofs}

\subsection{First spin flip: proof of Theorem \ref{th:firstspinflip}}  \label{section:firstspinflip}

By the definition, the time $T_1$ is the minimum of $N$ independent variables,
whose distributions are time--inhomogeneous exponential laws, whence 
\begin{equation}
\label{dtf}
\forall t\geq 0\qquad
P(T_1>t)\,=\,\exp\Bigg(-\sum_{1\leq i\leq N}
%\int_0^t\exp\big(\lambda(0,i)e^{-\alpha s}\big)\,ds\Bigg)\,.
	\int_0^t\exp\big(\lambda_{N,\beta}(i)e^{-\alpha s}\big)\,ds\Bigg)\,.
\end{equation}

We begin by performing the asymptotic expansion of the integral inside the exponential.
More precisely, let us define
$$
I(\gamma,t)\,=\,
\int_0^t\exp\big(-\gamma e^{-\alpha s}\big)\,ds\,.$$
Our first goal is to expand this integral in the limit as $\gamma$ goes to $+\infty$. 

A natural technique would be to
use Laplace's method of expansion. However, this integral is simple enough to be handled conveniently through an integration
by parts, as follows. We write
\begin{multline}
\label{ipp}
I(\gamma,t)\,=\,
\int_0^t
\Big(
\big(\gamma\alpha e^{-\alpha s}\big)
\exp\big(-\gamma e^{-\alpha s}\big)\Big)
\frac{1}{\gamma\alpha e^{-\alpha s}}
\,ds\cr
%\,=\,
%\int_0^t
%\Big(
%\big(\gamma\alpha e^{-\alpha s}\big)
%\exp\big(-\gamma e^{-\alpha s}\big)\Big)
%\frac{1}{\gamma\alpha e^{-\alpha s}}
%\,ds\cr
\,=\,
	\Bigg[
\exp\big(-\gamma e^{-\alpha s}\big)
\frac{1}{\gamma\alpha e^{-\alpha s}}
	\Bigg]_0^t
	-
\int_0^t
\exp\big(-\gamma e^{-\alpha s}\big)
\frac{1}{\gamma e^{-\alpha s}}
\,ds
\cr
\,=\,
\frac{ \exp\big(-\gamma e^{-\alpha t}\big)
}{\gamma\alpha e^{-\alpha t}}
-
\frac{ \exp\big(-\gamma \big)
}{\gamma\alpha }
	-
\frac{1}{\gamma }
\int_0^t
\exp\big(\alpha s-\gamma e^{-\alpha s}\big)
\,ds
	\,.
\end{multline}
Let us define
$$\forall \gamma\geq 0\quad\forall t\geq 0\qquad
F(\gamma,t)\,=\,
\frac{ \exp\big(-\gamma e^{-\alpha t}\big)
}{\gamma\alpha e^{-\alpha t}}
\,.$$
The function $F$ is the principal part of the asymptotic expansion of $I$ in the regime
where $\gamma$ tends to $+\infty$ and $t$ stays bounded.
In fact, from~\eqref{ipp}, we have on one hand
\begin{equation}
	\label{h1}
I(\gamma,t)\,\leq\,
F(\gamma,t)\,,
\end{equation}
and on the other hand
\begin{equation}
	\label{h2}
I(\gamma,t)\,\geq\,
F(\gamma,t)
-
	\frac{ e^{-\gamma } }{\gamma\alpha }
%F(\gamma,0)
%\frac{ \exp\big(-\gamma \big)
%}{\gamma\alpha }
	-
	\frac{ e^{\alpha t }
}{\gamma }
%\int_0^t
%\exp\big(-\gamma e^{-\alpha s}\big)
%\,ds
	I(\gamma,t)
	\,.
\end{equation}
Inequalities~\eqref{h1} and~\eqref{h2} yield that
\begin{align}
	\label{hh}
	0\, \leq\,
F(\gamma,t)
-
I(\gamma,t)
	&\, \leq\,
%F(\gamma,0)
	\frac{ e^{-\gamma } }{\gamma\alpha }
	+
	\frac{ e^{\alpha t }
}{\gamma }
%\int_0^t
%\exp\big(-\gamma e^{-\alpha s}\big)
%\,ds
	I(\gamma,t)
	\cr
	&\, \leq\,
%F(\gamma,0)
	\frac{ e^{-\gamma } }{\gamma\alpha }
	+
	\frac{ e^{\alpha t }
}{\gamma }
	F(\gamma,t)
%\int_0^t
%\exp\big(-\gamma e^{-\alpha s}\big)
%\,ds
	\,,
\end{align}
whence
\begin{equation}
	\label{ex}
I(\gamma,t)
\,=\,
F(\gamma,t)
	+O\Big(
	\frac{ e^{-\gamma } }{\gamma\alpha }
	+
	\frac{ e^{\alpha t } }{\gamma } F(\gamma,t)
	\Big)
	\,.
\end{equation}
We now give estimates for $P(T_1 > t)$ using \eqref{dtf} and the inequalities above. We begin with the case in which
\[
\lim \frac{\l_{N,\b}}{\ln N} < -1.
\]
{
%\color{red}
By \eqref{ex} we obtain, taking into account that $- \l_{N,\b} > \ln N$ for $N$ large and using \eqref{inhom_lambdas}:
\begin{multline}
	\label{coex}
	\sum_{1\leq i\leq N}
I( -\lambda_{N,\beta}(i) ,t)
\cr
	\,=\,\sum_{1\leq i\leq N}
F(-\lambda_{N,\beta}(i),t)
	+
	O\Big(
	\frac{ Ne^{ \l_{N,\b} } }{- \l_{N,\b}\alpha }
	+
	\frac{ e^{\alpha t } }{- \l_{N,\b} } 
	\sum_{1\leq i\leq N}
F(-\lambda_{N,\beta}(i),t)
	\Big)
\cr
	\,=\,
	\Big(
	\sum_{1\leq i\leq N}
F(-\lambda_{N,\beta}(i),t)
	\Big)
	\Big(1+O\Big(\frac{ e^{\alpha t } }{-\l_{N,\b} } \Big)\Big)
	+
	O\Big(
	\frac{ 1 }{\alpha  \ln N}
	\Big)\,.
\end{multline}
Using \eqref{inhom_lambdas}, we see that 
\[
F(-\lambda_{N,\beta}(i),t) = F(-\lambda_{N,\beta},t)(1+o(1)),
\]
where the term $o(1)$ can be chosen not dependent on $i$,
so
\begin{equation}
\label{T1est}
P(T_1>t) = \exp\left[ -N \frac{\exp\left(\l_{N,\b} e^{-\a t} \right)}{- \l_{N,\b} \a e^{-\a t}} \left( 1  + o(1) + O\left( \frac{e^{\a t}}{-\l_{N,\b}} \right) \right) + O\left(\frac{ 1 }{\alpha  \ln N} \right) \right].
\end{equation}
We now choose $t = t(N)$ so that $P(T_1>t)$ has a finite nonzero limit as $N \ra +\infty$. We set
\begin{equation}
\label{tchoice}
t = \frac{1}{\a} \ln \left(\frac{-\l_{N,\b}}{\ln N}\right) + u
\end{equation}
for some $u = u(N)$ that goes to zero as $N \ra +\infty$. Note that, with this choice,
\[
O\left( \frac{e^{\a t}}{-\l_{N,\b}} \right) = o(1).
\]
Inserting \eqref{tchoice} in \eqref{T1est} and using $e^{-\a u} = 1-\a u + o(1)$, we get
\begin{equation}
\label{wzzztf}
%\forall t\geq 0\qquad
	P\Big(T_1>
	\frac{1}{\alpha}\ln\Big(\frac{ -\lambda_{N,\beta} }{\ln N}\Big)
	+u
	\Big)\,=
	%\,\cr
	%\,P\big(\big|Y_{N,\beta}\big|>\ln\beta\big)\,+\,\cr
	\exp\Bigg(-
	%{\displaystyle \exp\Big( Y_{N,\beta} e^{-\alpha t } \Big)}
	\frac{\exp\big({ \alpha u\ln N+o(\ln N) }\big)}
	{ \alpha  {\ln N}	 }
	%{\displaystyle e^{Y_{N,\beta} e^{-\alpha t } }}
		 %\Big(1+O\Big(\frac{Y_{n,\beta}}{\beta}\Big)+ o({1}) \Big)
		 \big(1+ o({1}) \big) +  O\left(\frac{ 1 }{\alpha  \ln N} \right)
	%\Big(1+ O\Big( \frac{ e^{\alpha t } }{\beta } \Big) \Big)
	%\Big(1+O\Big(\frac{ e^{\alpha t } }{\beta } \Big)\Big)
	 %+ O\Big( \frac{N e^{-\beta } }{\beta\alpha }\Big)
	 \Bigg)\,.
\end{equation}
}
Now we choose $u$ so that 
\[
	\frac{\exp\big({ \alpha u\ln N+o(\ln N) }\big)}
	{ \alpha  {\ln N}	 }
\]
is bounded away form zero and infinity. Thus we take
\begin{equation}
\label{want}
	u\,=\,
%	\frac{1}{\alpha}
	\frac{\ln\ln N}{\alpha\ln N}
	+
	\frac{v}{\alpha\ln N}
	\,,
\end{equation}
for $v \in \R$.
Replacing $u$ with this expression   we get
\begin{equation}
\label{wzzzztf}
%\forall t\geq 0\qquad
	P\Big(T_1>
	\frac{1}{\alpha}\ln\Big(\frac{ -\lambda_{N,\beta} }{\ln N}\Big)
	+
	\frac{\ln\ln N}{\alpha\ln N}
	+
	\frac{v}{\alpha\ln N}
	\Big)\,=\,
	%P\big(\big|Y_{N,\beta}\big|>\ln\beta\big)\,+\,
	\exp\Big(-
	%{\displaystyle \exp\Big( Y_{N,\beta} e^{-\alpha t } \Big)}
	\displaystyle 
	\frac{1}{ \alpha }
	%\sqrt{\frac{2\pi\phi(x^*)}{ { \phi''(x^*) }}}
	\exp\big(v+
	%Y_{N,\beta} \Big(\frac{\ln N}{ -\lambda_{N,\beta} }\Big)
		  o({1}) \big)
	 \Big)\,.
\end{equation}
which completes the proof for the case 
\[
\lim \frac{\l_{N,\b}}{\ln N} < -1.
\]
Note that, setting $c := -\lim \frac{\l_{N,\b}}{\ln N}$, we have seen, in particular, that for $c>1$, 
\[
T_1 \ra \frac{1}{\a} \ln c
\]
in probability. Using the fact that, for each $t>0$, $P(T_1>t)$ is decreasing in $\l_{N,\b}$, by comparison it follows that $T_1 \ra 0$ in probability whenever $c\leq 1$.

\subsection{Covering time: proof of Theorem \ref{th:covering}} \label{section:covertime}

Before going into the details of the proof, we give an intuition of what happens during the covering.
Let $\overline{i} \in \{1,2,\ldots,N\}$ be such that 
\[
\tilde{\s}_i(0) = \left\{ \begin{array}{ll} -1 & \mbox{for } i \neq \overline{i} \\ 1 & \mbox{for } i = \overline{i}. \end{array} \right.
\]
The local field profile is given by
\begin{equation}\label{lambdainit}
\tilde{\l}_i(0) = \left\{ \begin{array}{ll}  2 \b + \l_i(T_1^{-}) & \mbox{for } i = \overline{i} \pm 1 \\ \l_i(T_1^{-}) & \mbox{otherwise}. \end{array} \right.
\end{equation}
where, by \eqref{lambdaT1bis}, 
\begin{equation}\label{past}
	\l_i(T_1^{-}) =  - \ln N + \ln \ln N + \ln \a + X_{N,\beta} + o(i,N,\beta).
\end{equation}
Note that the spins at $ \overline{i} \pm 1$ are likely to flip first, as, by \eqref{condbeta}, $2 \b \gg -\l_i(T_1^{-})$ with very high probability. Suppose that the first spin flip occurs at time $\tau_1$ for the spin $\overline{i} +1$. We have:
\begin{equation}\label{tau1}
\tilde{\l}_i(\tau_1) = \left\{ \begin{array}{ll}  
 \left[ 2 \b + \l_i(T_1^{-}) \right] e^{-\a \tau_1} & \mbox{for } i = \overline{i} \pm 1 \\  
 \l_i(T_1^{-}) e^{-\a \tau_1} + 2\b & \mbox{for } i = \overline{i}, \overline{i} +2 \\ \l_i(T_1^{-})e^{-\a \tau_1} & \mbox{otherwise}. \end{array} \right.
\end{equation}
In terms of the spin-flip rates $\tilde{r}_i(t) = \exp[-\tilde{\s}_i(t)\tilde{\l}_i(t)]$, note that $\tilde{r}_i(\tau_1) \leq 1$ with high probability for $i \neq \overline{i}-1, \overline{i} +2$, while 
\[
 \tilde{r}_{\overline{i}-1}(\tau_1) = \exp\left[ \left( 2 \b + \l_{\overline{i}-1}(T_1^{-}) \right) e^{-\a \tau_1} \right]
 \]
and
\[
 \tilde{r}_{\overline{i}+2}(t) = \exp\left[  \l_{\overline{i}+2}(T_1^{-}) e^{-\a \tau_1} + 2\b \right]\,
 \]
 are much larger than~$1$.
 It follows that the spins at $\overline{i}-1$ and $\overline{i} +2$ are likely to flip before the others. To have a better understanding, assume the spin at $\overline{i}-1$ flips first, at time $\tau_2$. The local field profile at time $\tau_2$ is then
 \begin{equation}\label{tau2}
 \tilde{\l}_i(\tau_2) = \left\{   \begin{array}{ll} \left[\l_i(T_1^{-}) e^{-\a \tau_1} + 2\b\right] e^{-\a(\tau_2 - \tau_1)} + 2 \b & \mbox{for } i = \overline{i} \\
  \left[ 2 \b + \l_i(T_1^{-}) \right] e^{-\a \tau_2} & \mbox{for } i = \overline{i}-1 \\
   \l_i(T_1^{-})e^{-\a \tau_2} + 2 \b & \mbox{for } i = \overline{i}-2 \\
  \left[ 2 \b + \l_i(T_1^{-}) \right]e^{-\a \tau_2} & \mbox{for } i = \overline{i}+1 \\
   \left[\l_i(T_1^{-}) e^{-\a \tau_1} + 2\b\right] e^{-\a(\tau_2 - \tau_1)} & \mbox{for } i = \overline{i}+2 \\
  \l_i(T_1^{-}) e^{-\a \tau_2} & \mbox{otherwise}.
    \end{array} \right.
\end{equation}
Again, we see that the spins at $\overline{i}\pm2$ are likely to flip first.
Thus, with high probability, as we will see in details next, a droplet of consecutive $+1$ spins forms. Denote by $\tau_n$ 
the time at which a droplet of length $n+1$ is formed, with $1 \leq n \leq N-3$. At time $\tau_n$ the local field in the {\em interior} 
of the droplet is bounded from below by $4 \b e^{-\a \tau_n} + \l_i(T_1^{-})$. In the internal boundary of the droplet the local field 
is bounded from below by $2 \b e^{-\a \tau_n} + \l_i(T_1^{-})$. In the external boundary of the droplet the local field satisfies
\begin{equation}\label{localtaun1}
\tilde{\l}_i(\tau_n) \in \left[2 \b + \l_i(T_1^{-}), 2 \b + \l_i(T_1^{-})e^{-\a \tau_n}\right]
\end{equation}
if $i$
 is the site neighbor of the last spin flipped, and 
\begin{equation}\label{localtaun2}
\tilde{\l}_i(\tau_n) \in \left[2 \b e^{-\a \tau_n} + \l_i(T_1^{-}), 2 \b + \l_i(T_1^{-})e^{-\a \tau_n}\right]
\end{equation}
for the other site. Note that the extremities of these intervals may be reversed in the case $\l_i(T_1^{-})>0$, which is unlikely (see \eqref{past}).
For all other sites the local field equals $\l_i(T_1^{-}) e^{-\a \tau_n}$. For $n = N-2$ the situation is slightly different, since there is only one site in the external boundary of the droplet. Denoting this site by $i^*$, we have
\[
	\tilde{\l}_{i^*}(\tau_{N-2}) \in \left[4 \b e^{-\a \tau_{N-2}} + \l_{i^*}(T_1^{-}), 4 \b + \l_{i^*}(T_1^{-})e^{-\a \tau_{N-2}}\right].
\]
This gives the intuition on how the local fields change according to the growth of the droplet of $+1$ spins. 
Notice that, with very large probability, the covering is performed (excluding the last flip) with a sequence of steps occurring with a rate of order
\begin{equation*}
 2 e^{2\b -\ln N +\ln \ln N + \ln \a + X_N + o(1)}={2\a \ln N \over N}     e^{2\b+X_N+o(1)}
\end{equation*}
and this is the intuitive reason to choose the time scaling
\begin{equation}
\label{ch_three_ratescovering} N (2 e^{2\b -\ln N +\ln \ln N + \ln \a + X_N })^{-1}=  {N^2 \over 2\a\ln N}e^{-2\b-X_N}
\end{equation}
appearing in \eqref{randscal}.\\

The strategy of the proof is to show that, in the limit, during the covering process only spins adjacent to the droplet will flip, 
and then to show that \eqref{ch_three_ratescovering} gives the correct time-scaling for the process where all undesired flips are suppressed.\\

{\bf Step 1: Probability of observing an undesired flip}\\ 
Let $\bar{\tau}$ be the time at which an ``undesired" flip occurs, i.e. the time at which we observe a flip of one of the spins that are {\em not} adjacent to the droplet. Our aim is to show that $P(\bar{\tau}\leq \tau_{N-1})$ converges to zero as $\beta,N\uparrow+\infty$. We estimate this probability conditioned to the event 
\[
A_N := \{- \ln N \leq \l_i(T_1^{-}) \leq - \ln N + 2 \ln\ln N: \, i=1, \ldots, N\},
\]
whose probability tends to one.  \\
Notice that, under $A_N$,   for $t\in[0,\tau_1[$, we have one positive spin with flipping rate at most
 \[
e^{\ln N},
\]
$N-3$ negative spins whose rates are at most
\[
e^{[-\ln N+2\ln\ln N]e^{-\alpha t}},
\]
and two negative spins, adjacent to the droplet of $+1$ spins, whose rates are at least 
 \[
 e^{[2\beta-\ln N]e^{-\alpha t}}.
 \]
Then, $P(\bar{\tau}\leq \tau_1|T_1){\bf 1}_{A_N}$ is bounded  by the probability that the first point of a Poisson process of time-dependent intensity 
\[
I_1(t):=  e^{\ln N }+(N-3)e^{[-\ln N+2\ln\ln N]e^{-\alpha t}}
\]
occurs before the first point of a point process with time-dependent intensity 
\[
J(t):= e^{\left[2 \b - \ln N\right]e^{-\a t}},
\]
where the two processes are independent.\\
%Denoting these first points by $X_1$ and $Y_1$ respectively, we conclude that 
%\be{undesired1}
%P(\bar{\tau}\leq \tau_1|T_1){\bf 1}_{A_N}\leq P(X_1<Y_1).
%\ee
Then, under $A_N\cap(\bar{\tau}>\tau_1)$,   for $t\in[\tau_1,\tau_2[$, we have a droplet consisting of two positive spins whose rates are at most
\[
e^{[-2\beta+\ln N]e^{-\alpha t}},
\]
$N-4$ negative spins whose rates are at most
\[
e^{[-\ln N+2\ln\ln N]e^{-\alpha t}},
\]
and two negative spins, adjacent to the droplet of $+1$ spins, whose rates are at least 
 \[
  e^{[2\beta-\ln N]e^{-\alpha t}}.
 \]
So, $P(\bar{\tau}\in] \tau_1,\tau_2]|T_1,(\bar{\tau}>\tau_1)){\bf 1}_{A_N}$ is bounded  by the probability that the first point of a Poisson process of time-dependent intensity 
\[
I_2(t):=2e^{[-2\beta+\ln N]e^{-\alpha t}}+(N-4)e^{[-\ln N+2\ln\ln N]e^{-\alpha t}}
\]
occurs before the first point of a point process with time-dependent intensity $J(t)$, 
where the two processes are independent.\\
Moreover, for any $k=3,\dots,N-2$, under $A_N\cap(\bar{\tau}>\tau_{k-1})$,   for $t\in[\tau_{k-1},\tau_k[$, we have a droplet consisting of $k$ positive spins: two of them with rates at most
\[
e^{[-2\beta+\ln N]e^{-\alpha t}},
\]
and $k-2$ of them with rates at most
\[
e^{[-4\beta+\ln N]e^{-\alpha t}}.
\]
The system will also present $N-k-2$ negative spins whose rates are at most
\[
e^{[-\ln N+2\ln\ln N]e^{-\alpha t}},
\]
and two negative spins, adjacent to the droplet of $+1$ spins, whose rates are at least 
 \[
  e^{[2\beta-\ln N]e^{-\alpha t}}.
 \]
Then, for any $k=3,\dots,N-2$, $P(\bar{\tau}\in] \tau_{k-1},\tau_{k}]|T_1,(\bar{\tau}>\tau_{k-1})){\bf 1}_{A_N}$ is bounded  by the probability that the first point of a Poisson process of time-dependent intensity 
\[
I_k(t):=2 e^{[-2\beta+\ln N]e^{-\alpha t}}+(k-2) e^{[-4\beta+\ln N]e^{-\alpha t}}+(N-k-2)e^{[-\ln N+2\ln\ln N]e^{-\alpha t}}
\]
occurs before the first point  of a point process with time-dependent intensity 
 $J(t)$, 
where the two processes are independent.\\
Finally, under $A_N\cap(\bar{\tau}>\tau_{N-2})$,   for $t\in[\tau_{N-2},\tau_{N-1}[$, we have a droplet comprised by $N-1$ positive spins: two of them with rates at most
\[
e^{[-2\beta+\ln N]e^{-\alpha t}},
\]
and $N-3$ of them with rates at most
\[
e^{[-4\beta+\ln N]e^{-\alpha t}}.
\]
The system also presents one negative spin, with rate at least
\[
  e^{[4\beta-\ln N]e^{-\alpha t}}.
\]
Then, $P(\bar{\tau}\in] \tau_{N-2},\tau_{N-1}]|T_1,(\bar{\tau}>\tau_{N-2})){\bf 1}_{A_N}$ is bounded  by the probability that the first point of a Poisson process of time-dependent intensity 
\[
I_{N-1}(t):=2 e^{[-2\beta+\ln N]e^{-\alpha t}}+(N-3)e^{[-4\beta+\ln N]e^{-\alpha t}}
\]
occurs before the first point of a point process with time-dependent intensity 
 $J(t)$,
where the two processes are independent.\\
By the analysis above, we can consider a family of Poisson processes $\{ \zeta_k\}_{k=1}^{N-1}$ with time-dependent intensities $\{I_k(t)\}_{k=1}^{N-1}$ and a Poisson process $\eta$ with time-dependent intensity $J(t)$. We also assume that, for any $k=1,\dots,N-1$, the process $\zeta_k$ and the process $\eta$ are independent. Let us denote with $X_k$ the first point of the process $\zeta_k$ for any $k=1,\dots, N-1$, and with $Y$ the first point of the process $\eta$. In this way, we deduce that
\begin{align*}
P(\bar{\tau}\leq \tau_{N-1}|T_1){\bf 1}_{A_N}&=P(\bar{\tau}\leq \tau_1|T_1){\bf 1}_{A_N}+\sum_{k=2}^{N-1} P(\bar{\tau}\in]\tau_{k-1},\tau_{k}]|T_1){\bf 1}_{A_N}\\
&\leq P(\bar{\tau}\leq \tau_1|T_1){\bf 1}_{A_N}+\sum_{k=2}^{N-1} P(\bar{\tau}\leq\tau_{k}|T_1,\bar{\tau}>\tau_{k-1}){\bf 1}_{A_N}\\
&\leq \sum_{k=1}^{N-1} P(X_k<Y).
\end{align*}
Let us fix $
T_m= e^{-d\beta}$ with $d$ a positive constant such that $2c+d<2$. We have that, for any $k=1,\dots,N-1$,
{
%\color{red}
\begin{align*}
P(X_k<Y)&\leq P(X_k<Y\leq T_m)+P(Y>T_m)\\
&\leq   \frac{ P(X_k<Y\leq T_m) }{ P(X_k\leq T_m, Y\leq T_m)} +P(Y>T_m)\\
&= P(X_k<Y\leq T_m) | X_k\leq T_m, Y\leq T_m)+P(Y>T_m)\\
&\leq  P(X_k<Y  | X_k\leq T_m, Y\leq T_m) +P(Y>T_m)
\end{align*}
}
where we used the independence of $X_k$ and $Y$, hence
\begin{align*}
P(\bar{\tau}\leq \tau_{N-1}|T_1){\bf 1}_{A_N}&\leq \sum_{k=1}^{N-1} P(X_k<Y|Y\leq T_m,X_k\leq T_m) +(N-1) P(Y>T_m).
\end{align*}
Notice that, conditioned to $(X_k\leq T_m,Y\leq T_m)$ the distribution of $X_k$ is stochastically bigger than the one of an exponential random variable of parameter $I_k(T_m)$ and the distribution of $Y$ is stochastically smaller than the one of an exponential r.v. of parameter $J(T_m)$. This means that, for any $k=1,\dots,N-1$, 
\[
 P(X_k<Y|Y\leq T_m,X_k\leq T_m)\leq {I_k(T_m)\over I_k(T_m)+J(T_m)},
\]
so we get 
\begin{equation}
\label{maggiorazione} P(\bar{\tau}\leq \tau_{N-1}|T_1){\bf 1}_{A_N}\leq \sum_{k=1}^{N-1} {I_k(T_m)\over I_k(T_m)+J(T_m)} +(N-1) P(Y>T_m).
\end{equation}
Let us consider the second term on the right-hand side of \eqref{maggiorazione}: 
\begin{align*}
(N-1)P(Y>e^{-d\beta})&\leq (N-1)\exp\left[ -\int_0^{e^{-d\beta}} e^{[2\beta-\ln N]e^{-\alpha t}} dt\right]\\
&\leq \exp\left[\ln N - e^{- d\beta}e^{[2\beta-\ln N]e^{-\alpha e^{- d\beta}}}  \right]\\
&\approx\exp\left[\ln N - e^{ (2-d)\beta-\ln N -\a e^{-d\b} [2\b-\ln N] + o(\b e^{-d\b})}  \right]\underset{\beta,N\uparrow+\infty}{\longrightarrow}0,
\end{align*}
thanks to \eqref{condbeta} and the fact that $c+d<2$. \\
Consider now the first term in the right-hand side of \eqref{maggiorazione}. The following limits hold: 
\begin{multline*}
 {I_1(T_m)\over I_1(T_m)+J(T_m)}=  { e^{\ln N}+(N-3)e^{[-\ln N+2\ln\ln N]e^{-\alpha T_m}}\over  e^{\ln N}+(N-3)e^{[-\ln N+2\ln\ln N]e^{-\alpha T_m}}+ e^{\left[2 \b - \ln N\right]e^{-\a T_m}}}\underset{\beta,N\uparrow+\infty}{\longrightarrow}0,
 \end{multline*}
 \begin{multline*}
{I_2(T_m)\over I_2(T_m)+J(T_m)}=\\ {2e^{[-2\beta+\ln N]e^{-\alpha T_m}}+(N-4)e^{[-\ln N+2\ln\ln N]e^{-\alpha T_m}}\over 2e^{[-2\beta+\ln N]e^{-\alpha T_m}}+(N-4)e^{[-\ln N+2\ln\ln N]e^{-\alpha T_m}}+ e^{\left[2 \b - \ln N\right]e^{-\a T_m}}}\underset{\beta,N\uparrow+\infty}{\longrightarrow}0,
\end{multline*}
\begin{multline*}
{I_{N-1}(T_m)\over I_{N-1}(T_m)+J(T_m)}=\\ {2 e^{[-2\beta+\ln N]e^{-\alpha T_m}}+(N-3)e^{[-4\beta+\ln N]e^{-\alpha T_m}}\over 2 e^{[-2\beta+\ln N]e^{-\alpha T_m}}+(N-3)e^{[-4\beta+\ln N]e^{-\alpha T_m}}+ e^{\left[2 \b - \ln N\right]e^{-\a T_m}}}\underset{\beta,N\uparrow+\infty}{\longrightarrow}0.
\end{multline*}
Moreover, it holds that
\begin{align*}
&\sum_{k=3}^{N-2} {I_k(T_m)\over I_k(T_m)+J(T_m)}\\
&=\sum_{k=3}^{N-2} {2 e^{[-2\beta+\ln N]e^{-\alpha T_m}}+(k-2)e^{[-4\beta+\ln N]e^{-\alpha T_m}}+(N-k-2)e^{[-\ln N+2\ln\ln N]e^{-\alpha T_m}}\over I_k(T_m) +J(T_m)}\\
&\leq { N 2 e^{[-2\beta+\ln N]e^{-\alpha T_m}}\over J(T_m)}+ {N^2 e^{[-4\beta+\ln N]e^{-\alpha T_m}}\over J(T_m)}+\sum_{k=3}^{N-2}{(N\!-\!k-2)e^{[-\ln N+2\ln\ln N]e^{-\alpha T_m}}\over I_k(T_m)+J(T_m)},
\end{align*}
where
\[
{N2 e^{[-2\beta+\ln N]e^{-\alpha T_m}}\over J(T_m)}={2 e^{[-2\beta+\ln N]e^{-\alpha T_m}+\ln N}\over e^{\left[2 \b - \ln N\right]e^{-\a T_m}}} \underset{\beta,N\uparrow+\infty}{\longrightarrow}0,
\]
\[
{N^2 e^{[-4\beta+\ln N]e^{-\alpha T_m}}\over J(T_m)}={e^{[-4\beta+\ln N]e^{-\alpha T_m}+2\ln N}\over e^{\left[2 \b - \ln N\right]e^{-\a T_m}} }\underset{\beta,N\uparrow+\infty}{\longrightarrow}0,
\]
while
\begin{align*}
&\sum_{k=3}^{N-2}{(N-k-2)e^{[-\ln N+2\ln\ln N]e^{-\alpha T_m}}\over I_k(T_m)+J(T_m)}\\
&\leq\sum_{k=1}^{N-5} {ke^{[-\ln N+2\ln\ln N]e^{-\alpha T_m}}\over ke^{[-\ln N+2\ln\ln N]e^{-\alpha T_m}}+J(T_m)}\\
&= \sum_{k=1}^{N-5} {k\over k+ e^{[2\b -2\ln\ln N]e^{-\a T_m} }} \\
&\leq \int_{0}^N {x\over x+e^{[2\b -2\ln\ln N]e^{-\a T_m} }}dx\\
&= N + e^{[2\b -2\ln\ln N]e^{-\a T_m} } \ln\left[{e^{[2\b -2\ln\ln N]e^{-\a T_m} }\over N+ e^{[2\b -2\ln\ln N]e^{-\a T_m} }}\right].
\end{align*}
Notice that
\begin{align*}
&\lim_{\b,N\uparrow+\infty} \left(N + e^{[2\b -2\ln\ln N]e^{-\a T_m} } \ln\left[{e^{[2\b -2\ln\ln N]e^{-\a T_m} }\over N+ e^{[2\b -2\ln\ln N]e^{-\a T_m} }}\right]\right)\\
&=\lim_{\b,N\uparrow+\infty}\left(N+ e^{[2\b -2\ln\ln N]e^{-\a T_m} } \ln\left[ 1-{N\over N+e^{[2\b -2\ln\ln N]e^{-\a T_m} }}\right]\right)\\
&=\lim_{\b,N\uparrow+\infty}\!\left(\!{N}\!+\!e^{[2\b -2\ln\ln N]e^{-\a T_m} }\left(\!{1\over2}\!\left({N\over N+e^{[2\b -2\ln\ln N]e^{-\a T_m} }}\right)^2\!\!\!-{N\over N+e^{[2\b -2\ln\ln N]e^{-\a T_m} }}\! \right)\!\right)\!\\
&=\lim_{\b,N\uparrow+\infty}\left({N^2\over e^{[2\b -2\ln\ln N]e^{-\a T_m}}}+{N^2  e^{[2\b -2\ln\ln N]e^{-\a T_m}}\over 2 e^{[4\b -4\ln\ln N]e^{-\a T_m} }}\right)\\
&=\lim_{\b,N\uparrow+\infty} {3\over 2}\exp\Big[2\ln N - [2\b -2\ln\ln N]e^{-\a T_m}\Big]\\
&=\lim_{\b,N\uparrow+\infty} {3\over 2}\exp\Big[2\ln N - 2\b + 2\ln\ln N + \a e^{-d\beta} (2\b -2\ln\ln N)\Big]=0,
\end{align*}
%{\small
%\[
%\lim_{\b,N\uparrow+\infty} \left(N + e^{[2\b -2\ln\ln N]e^{-\a T_m} } \ln\left[{e^{[2\b -2\ln\ln N]e^{-\a T_m} }\over N+ e^{[2\b -2\ln\ln N]e^{-\a T_m} }}\right]\right)=
%\]
%\[
%\lim_{\b,N\uparrow+\infty}\left(N+ e^{[2\b -2\ln\ln N]e^{-\a T_m} } \ln\left[ 1-{N\over N+e^{[2\b -2\ln\ln N]e^{-\a T_m} }}\right]\right)=
%\]
%\[
%\lim_{\b,N\uparrow+\infty}\!\left(\!{N}\!+\!e^{[2\b -2\ln\ln N]e^{-\a T_m} }\left(\!{1\over2}\!\left({N\over N+e^{[2\b -2\ln\ln N]e^{-\a T_m} }}\right)^2\!\!\!-{N\over N+e^{[2\b -2\ln\ln N]e^{-\a T_m} }}\! \right)\!\right)\!=
%\]
%\[
%\lim_{\b,N\uparrow+\infty}\left({N^2\over e^{[2\b -2\ln\ln N]e^{-\a T_m}}}+{N^2  e^{[2\b -2\ln\ln N]e^{-\a T_m}}\over 2 e^{[4\b -4\ln\ln N]e^{-\a T_m} }}\right)=
%\]
%\[
%\lim_{\b,N\uparrow+\infty} {3\over 2}\exp\Big[2\ln N - [2\b -2\ln\ln N]e^{-\a T_m}\Big]=
%\]
%\[
%\lim_{\b,N\uparrow+\infty} {3\over 2}\exp\Big[2\ln N - 2\b + 2\ln\ln N + \a e^{-d\beta} (2\b -2\ln\ln N)\Big]=0,
%\]}
thanks to \eqref{condbeta}.\\
All these considerations imply that the probability that an undesired spin flip occurs before the droplet of $+1$ spins invades the whole space converges to zero: 
\begin{align*}
P(\bar{\tau}\leq \tau_{N-1})&=E( P(\bar{\tau}\leq \tau_{N-1}|T_1){\bf 1}_{A_N})+P(A_N^c)\\
&\leq E\left( P(\bar{\tau}\leq \tau_1|T_1){\bf 1}_{A_N}+\sum_{k=2}^{N-1} P(\bar{\tau}\leq\tau_{k}|T_1,(\bar{\tau}>\tau_{k-1})){\bf 1}_{A_N}\right)+P(A_N^c)\\
&\leq \sum_{k=1}^{N-1} P(X_k<Y) + P(A_N^c)\\
&\leq \sum_{k=1}^{N-1} {I_k(T_m)\over I_k(T_m)+J(T_m)} +(N-1) P(Y>T_m)+P(A_N^c)\underset{\beta,N\uparrow+\infty}{\longrightarrow}0.
\end{align*}
\bigskip

{\bf Step 2: Time-scaling for the covering process}\\
Using the graphical construction we can couple the process $(\tilde{\s},\tilde{\l})$ with the process $(\hat{\s}, \hat{\l})$ obtained by suppressing all undesired spin flip; in other words $\tilde{\s}(0) = \hat{\s}(0)$, and
\[
\hat{\l}_i(t) = \left\{ \begin{array}{ll} \tilde{\l}_i(t) & \mbox{if } \hat{\s}_i(t) = -1, \mbox{ and }  \hat{\s}_j(t) = 1 \mbox{ for at least one } j \in \{i-1,i+1\} \\ 0 & \mbox{otherwise} \end{array} \right.
\]
By the estimates above, we have
\begin{equation}\label{hatapprox}
\lim_{\b,N \uparrow +\infty} P(\tilde{\s}(t) = \hat{\s}(t) \mbox{ for } t \in [0,T_c]) = 1.
\end{equation}
Thus, to compute the distribution of $T_c$, we can use the process $\hat{\s}$ in place of $\tilde{\s}$. Note that the times $\tau_n$ introduced above are well defined for the process $(\hat{\s}, \hat{\l})$. Moreover, $T_c = \tau_{N-1}$ on the event $\{\tilde{\s}(t) = \hat{\s}(t) \mbox{ for } t \in [0,T_c]\}$. Using the same estimate as in Step 1, for $n=1,\ldots,N-1$, on $A_N$ we have
\[
P \left(\tau_n - \tau_{n-1} > \frac{e^{-d\b}}{N} \big| T_1\right) \leq \exp\left[ -\frac{e^{-d\b}}{N}e^{\left[2 \b - \ln N\right]e^{-\a \frac{e^{-d\b}}{N}}}\right] 
\]
 which implies, defining
 \begin{equation}\label{BN}
 B_N:= \left\{ \tau_n - \tau_{n-1} \leq  \frac{e^{-d\b}}{N}: \, n=1,\ldots,N-1\right\},
 \end{equation}
 the estimate
\begin{align}
\label{estBN}
P(B_N^c | T_1){\bf 1}_{A_N} &\leq N \exp\left[ -\frac{e^{-d\b}}{N}e^{\left[2 \b - \ln N\right]e^{-\a \frac{e^{-d\b}}{N}}}\right]\nonumber \\
&= \exp\left[ \ln N - e^{-d\b-\ln N +[2\b -\ln N]e^{-\a\frac{e^{-d\b}}{N}}}  \right]\nonumber \\
&= \exp \left[ \ln N - e^{(2-d)\beta-2\ln N + \a \frac{e^{-d\b}}{N}[2\b-\ln N]+o\left( \b\frac{e^{-d\b}}{N}\right)  }\right]\to 0
\end{align}
as $\b,N \uparrow +\infty$ since $2c+d<2$. This estimate, together with \eqref{hatapprox}, gives also
\begin{equation}\label{boundTc}
\lim_{\b,N \uparrow +\infty} P(T_c > e^{-d\b}|T_1){\bf 1}_{A_N} = \lim_{\b,N \uparrow +\infty} P(\tau_{N-1} > e^{-d\b}|T_1){\bf 1}_{A_N}  \ra 0\,.
\end{equation}
Having all these preliminary estimates, we now aim at giving sharp estimates on the distribution of $\tau_{N-1}$. The key idea is to write
\[
\tau_{N-1} = \tau_1 + \sum_{k=2}^{N-1} \left(\tau_k - \tau_{k-1}\right),
\]
and show that the random variables in the sum above are nearly independent and identically distributed.  We  define
\[
L_N :=  - \ln N + \ln \ln N + \ln \a + X_N,
\]
so that $\l_i(T_1^{-}) = L_N + o_i(1)$, where for each $\e>0$
\[
P\left(\max_{i=1,\ldots,N} |o_i(1)|>\e \right) \ra 0
\]
as $\b,N \uparrow +\infty$, see \eqref{lambdaT1bis} . By \eqref{lambdainit}, 
\[
P(\tau_1>t | T_1) = P(\min(X_+, X_-) > t  | T_1),
\]
where $X_+, X_-$ are random variables which are independent conditionally to ${T_1}$, and 
\[
P\left(X_{\pm} > t |T_1\right) = \exp\left[ - \int_0^t \left( e^{(2\b + \l_{\overline{i} \pm 1}(T_1^-))e^{-\a s}} \right) ds \right].
\]
Therefore
\begin{equation}\label{disttau1}
\begin{split}
P(\tau_1>t |T_1) & = \exp\left[  - \int_0^t \left( e^{(2\b + \l_{\overline{i} + 1}(T_1^-))e^{-\a s}} \right) ds   - \int_0^t \left( e^{(2\b + \l_{\overline{i} - 1}(T_1^-))e^{-\a s}} \right) ds \right] \\ & = 
\exp\left[  - 2\int_0^t \left( e^{(2\b + L_N + o(1))e^{-\a s}} \right) ds  \right]
\end{split}
\end{equation}
where $o(1)$ denotes a $T_1$-measurable random variable which goes to zero in probability. More generally, using \eqref{localtaun1} and \eqref{localtaun2}, for $2 \leq n \leq N-3$,
\begin{equation}\label{disttaun1}
P(\tau_{n+1} - \tau_n>t | \tau_n, \tau_{n-1},\ldots, \tau_1, T_1) \geq \exp\left[ - 2 \int_0^t e^{\left(2\b + ( L_N + o(1))e^{-\a \tau_n} \right) e^{-\a s} } ds\right]
\end{equation}
and
\begin{equation}\label{disttaun2}
P(\tau_{n+1} - \tau_n>t | \tau_n, \tau_{n-1},\ldots, \tau_1, T_1) \leq \exp\left[ - 2 \int_0^t e^{\left(2\b e^{-\a \tau_n}  +  L_N + o(1)\right) e^{-\a s} } ds\right].
\end{equation}
The case $n = N-2$ is similar, the $2$ multiplying the $\int_0^t$ must be removed and $2\b$ replaced by $4\b$.
%As in \eqref{BN}, set
%\[
%B_n :=  \left\{ \tau_k - \tau_{k-1} \leq  \frac{e^{-\b}}{N}: \, k=1,\ldots,n-1\right\}
%\]
%for $n \leq N$. 
%Define $\xi_n := \left[\tau_n - \tau_{n-1}\right] {\bf 1}_{B_{n+1}}$. Note that
%\[
%\tau_{N-1} = \xi_1 + \xi_{N-1}
%\]
%on $B_N$. 
By \eqref{disttaun1}, on $B_{N} \cap A_N$ we have, for $n=1,\ldots,N-3$ and the above correction for $n=N-2$,
\begin{equation}\label{lowerxi}
P(\tau_{n+1} - \tau_n>t  | \tau_n, \tau_{n-1},\ldots, \tau_1, T_1) \geq  \exp\left[ - 2 t e^{2\b 
	%e^{-\a e^{-d\b}}  
	+  L_N + o(1) }\right] = P(Y_n > t|T_1)
\end{equation}
where $(Y_1,\ldots,Y_{N-2})$, conditioned to $T_1$ are independent and have exponential distribution with mean
\[
E(Y_n|T_1) = \frac12 e^{-\left[2\b 
%e^{-\a e^{-d\b}}  
+  L_N + o(1) \right]}
\]
for $n \leq N-3$, while
\[
E(Y_{N-2}|T_1) =  e^{-\left[4\b 
%e^{-\a e^{-d\b}}  
+  L_N + o(1)\right] }.
\]
Thus, by Lemma \ref{lemma:sum}, 
{
%\color{red}
which is stated and proved below,
}
the following inequality holds on $A_N$, for every $t>0$:
\begin{equation}\label{lbtauN}
P(\tau_{N-1}>t|T_1) \geq P(Y_0 + \cdots + Y_{N-2} >t | T_1) - NP(B_N^c |T_1).
\end{equation}
Since, for $j \leq N-3$
\begin{equation}\label{meanlower}
E(Y_j|T_1) = \frac{N}{2\a \ln N}e^{-2\b- X_N}(1+o(1)).
\end{equation}
the Law of Large Numbers for $(Y_n)$ gives
\[
\lim_{\b,N \uparrow +\infty} P\left(\frac{Y_0 + \cdots + Y_{N-2}}{ \frac{N^2}{\a \ln N}e^{-2\b- X_N}}>1-\e \right) =1
\]
for every $\e>0$. Inserting this in \eqref{lbtauN}, using the fact that by \eqref{estBN}
\[
NP(B_N^c |T_1){\bf 1}_{A_N} \ra 0,
\]
$P(A_N^c) \ra 0$ and $T_c = \tau_{N-1}$ on $A_N$, we obtain
\begin{equation}\label{finlbtauN}
	\forall\e>0\qquad
\lim_{\b,N \uparrow +\infty} P\left(\frac{T_c}{ \frac{N^2}{\a \ln N}e^{-2\b- X_N}}>1-\e \right) =1.
\end{equation}
{
%\color{red}
To obtain a corresponding upper bound, define $\xi_n:= \left[\tau_{n+1} - \tau_n \right] {\bf 1}_{B_N}$. 
By \eqref{disttaun2} and observing that, by definition of $B_N$, $\xi_n \leq \frac{e^{-d\b}}{N}$ almost surely , 
\[
P(\xi_n>t| \tau_n, \tau_{n-1},\ldots, \tau_1, T_1) \leq  \exp\left[ - 2 t e^{\left(2\b e^{-\a e^{-d\b}}  +  L_N + o(1)\right) e^{-\a \frac{e^{-d \b}}{N}} } \right] = P(Z_n>t|T_1),
\]
where $Z_1, \ldots, Z_{N-2}$ are, conditionally to $T_1$, independent, exponentially distributed with mean
\[
\frac12 e^{-\left(2\b e^{-\a e^{-d\b}}  +  L_N + o(1)\right) e^{-\a \frac{e^{-d\b}}{N}} } =  \frac{N}{2\a \ln N}e^{-2\b- X_N}(1+o(1)).
\]
}
Using Lemma \ref{lemma:sum} as above and observing that $\xi_1 + \cdots + \xi_{N-1} = \tau_{N-1} = T_c$ on $B_N \cap A_N$, we obtain
\begin{equation}\label{finubtauN}
	\forall\e>0\qquad
\lim_{\b,N \uparrow +\infty} P\left(\frac{T_c}{ \frac{N^2}{\a \ln N}e^{-2\b- X_N}}<1+\e \right) =1\,.
\end{equation}
Together with \eqref{finlbtauN}, this completes the proof.

In the proof of Theorem \ref{th:covering}, we showed that the sequence of times taken to perform each step in the covering can be stochastically dominated, both from above and from below, by two different families of i.i.d. random variables obeying the same Law of Large Numbers. This was sufficient to conclude the proof thanks to the following technical lemma.
\begin{lemma} \label{lemma:sum}
Let $X=(X_n)_{n=1}^N$ and $Y=(Y_n)_{n=1}^N$ be two random vectors, such that $X$ is adapted to a filtration $(\mathcal{F}_n)_{n=1}^N$, and $Y$ has independent components. Define $S_n^X:= X_1 + \cdots + X_n$, and similarly $S_n^Y$. Assume there is an event $B$ such that for every $t \in \R$ and $n \in \{1,\ldots,N\}$ and $\omega \in B$
\begin{equation}\label{lowerin}
P(X_n>t|\mathcal{F}_{n-1})(\omega) \leq P(Y_n >t).
\end{equation}
Then for every $t \in \R$ and $n \in \{1,\ldots,N\}$
\begin{equation}\label{finlowerin}
P(S_n^X > t) \leq P(S_n^Y > t)+ n P(B^c)\,.
\end{equation}
Similarly, if \eqref{lowerin} is replaced by
\begin{equation}\label{upperin}
P(X_n>t|\mathcal{F}_{n-1})(\omega) \geq P(Y_n >t),
\end{equation}
then we have
\begin{equation}\label{finupperin}
P(S_n^X > t) \geq P(S_n^Y > t)- n P(B^c)\,.
\end{equation}
\end{lemma}
\begin{proof}
We prove the desired statement by induction on $n$. For $n=1$ there is nothing to prove. Note that, without loss of generality, we can assume $Y$ to be independent of $X$. 
{
%\color{red}
By assumption, for every $s,t \in \R$ and $\o \in B$
\[
P(X_{n+1} > t-s|\mathcal{F}_n)(\o) \leq P(Y_{n+1} > t-s|\mathcal{F}_n)(\o).
\]
This inequality still holds if we replace $s$ by the $\mathcal{F}_n$-measurable random variable $S_n^X$.
Thus we have, on $B$,
\[
P(S_{n+1}^X>t|\mathcal{F}_n) 
 \leq P(Y_{n+1} + S_n^X > t|\mathcal{F}_n),
\]
so that 
\[
P(S_{n+1}^X>t) \leq P(Y_{n+1} + S_n^X > t) + P(B^c).
\]
}
On the other hand, denoting by $\mathcal{L}_{Y_{n+1}}$ the law of $Y_{n+1}$, using the inductive assumption we obtain
\begin{align*}
P(Y_{n+1} + S_n^X > t) &= \int P(S_n^X > t-y) \mathcal{L}_{Y_{n+1}}(dy) \\& \leq  \int P(S_n^Y > t-y) \mathcal{L}_{Y_{n+1}}(dy) + nP(B^c) = P(S_{n+1}^Y > t)+ nP(B^c) .
\end{align*}
The proof with the reversed inequalities is identical.

\end{proof}

\subsection{Oscillating behavior: proof of Theorem \ref{thm:macroscopic_oscillations}}  \label{oscillation}

Let us start with some preliminary considerations: define the stopping times \[
\tau_{\tilde{m}}=\inf\{t>0\:|\: \tilde{m}_N(t)= +1\},\ \ \ \ 
\tau_{\tilde{x}}=\inf\{t>0\:|\: \tilde{x}(t)=+1\}.\]
For simplicity of notations and readability of the proof we only prove that the process $(\tilde{m}_N(t\wedge \tau_{\tilde{m}}))_{t\in[0,T]}$ converges to $(\tilde{x}(t\wedge \tau_{\tilde{x}}))_{t\in[0,T]}$ as $\beta,N\uparrow +\infty$: the result can be extended for any finite number of iterations (the same idea as Theorem \ref{thm:many_stopping}).
% and this is sufficient since $$P\left(\sum_{i= 1}^{+\infty} Z_i<T\right)=0.$$
Moreover, consider the process $(\hat{\underline{\sigma}}(t),\hat{\underline{\lambda}}(t))_{t\geq 0}$ defined at the beginning of Step 2 in 
the proof of Theorem \ref{th:covering}, for which all ``undesired" flips are suppressed: coupling this process with the original one $(\underline{\sigma}(t),\underline{\lambda}(t))_{t\geq 0}$ via the graphical construction implies that
\[ \lim_{\beta,N\uparrow +\infty} P(\hat{\sigma}_i(t)=\sigma_i(t) \mbox{ for } i\in\Lambda_N, \:t\in[0,T_1+T_c])=1,\]
hence we can prove the result using the total magnetization corresponding to the process $\underline{\hat{\sigma}}$ instead of $\underline{\sigma}$. Now we are ready for the proof.\\
Since the amplitude of the jumps of the process $(\tilde{m}_N(t))_{t\in[0,T]}$ converges to zero, then, 
{
%\color{red}
provided its weak limit
 $(\tilde{m}(t))_{t\in[0,T]}$ exists,
 }
 it holds $P(\tilde{m}\in\mathcal{C}([0,T],\mathbb{R})=1$ (see \cite{B99}, Theorem 13.4). This implies that the convergence can be studied on the space $\mathcal{D}([0,T],\mathbb{R})$ endowed with the uniform metric and topology (see for example Lemma 1.6.4 in \cite{S04}). With this choice, on $\mathcal{M}_{1}(\mathcal{D}([0,T],\mathbb{R})$ the \textit{Wasserstein distance} $\mathcal{W}_1$ 
%(see Definition \ref{WASSERSTEIN}) 
reads
\begin{equation}
\label{wasserstein}
\mathcal{W}_1(\mu,\nu)=\inf_{\gamma\in\Gamma(x,y)} \int_{\mathcal{D}\times\mathcal{D}} ||x-y||_{\infty}d\gamma(x,y)
\end{equation}
where $\Gamma(\mu,\nu)$ is the set of all possible couplings of $\mu$ and $\nu$.\\
Consider the events 
\[A_N=\{-\ln N \leq \l_i(T_1^-)\leq -\ln N +2\ln \ln N \: :\: i=1,\dots,N\},\]
\[B_N=\left\{ \tau_n-\tau_{n-1}>{e^{-d\b}\over N}\: : \: n=1,\dots, N-1\right\}
\]
 introduced while proving Theorem \ref{th:covering}. The strategy of the proof is to use the graphical construction to couple $(\tilde{m}_N(t))_{t\in[0,T]}$ with two processes $(\tilde{m}^+_N(t))_{t\in[0,T]}$  and $(\tilde{m}^-_N(t))_{t\in[0,T]}$,  both converging to $(\tilde{x}(t))_{t\in[0,T]}$, in such a way
 that under $A_N\cap B_N$ it holds
\[\tilde{m}^-_N(t)\leq \tilde{m}_N(t)\leq \tilde{m}_N^+(t) \:\:\:\:\: \mbox{ for any } t\in[0,T\wedge \tau_{\tilde{m}}].\]
Roughly speaking, $\tilde{m}^+_N$ (respectively $\tilde{m}^-_N$) represents the time-scaled magnetization of a spin system $\underline{\eta}_N^+$ (resp. $\underline{\eta}_N^-$) in which all undesired spins are suppressed and, after time $T_1$, each flip is performed with a higher (resp. lower) rate with respect to $\hat{\underline{\sigma}}$. Moreover, the rates for $\underline{\eta}_N^+$ and $\underline{\eta}_N^-$ have to be chosen in such a way that both $\tilde{m}^+_N$ and $\tilde{m}^-_N$ converge to $\tilde{x}$.\\
On the probability space already defined for the graphical construction, define the spin processes $\underline{\eta}^+$ and $\underline{\eta}^-$ in the following way:
\[ \eta^+_i(t)=\eta^-_i(t)=\sigma_i(t), \:\:\:\:\: t\in[0,T_1],\]
while, after $T_1$, the local fields for $\underline{\eta}^+$ and $\underline{\eta}^-$ are defined as:
{\small
\[
\lambda^+_i(t) = \left\{ \begin{array}{ll} 2\b+L_Ne^{-\a e^{-d\b}} & \mbox{if } \eta^+_i(t) = -1, \mbox{ and }  \eta^+_j(t) = 1 \mbox{ for at least one } j \in \{i-1,i+1\} \\ 0 & \mbox{otherwise} \end{array} \right.
\]
\[
\lambda^-_i(t) = \left\{ \begin{array}{ll} 2\b e^{-\a e^{-d\b}}+L_N & \mbox{if } \eta^-_i(t) = -1, \mbox{ and }  \eta^-_j(t) = 1 \mbox{ for at least one } j \in \{i-1,i+1\} \\ 0 & \mbox{otherwise} \end{array} \right.
\]}
where, as in the proof of Theorem \ref{th:covering},
\[ L_N=\lambda_i(T_1^-)=-\ln N+\ln\ln N+\ln \alpha + X_N+o(1).\]
Of course, by construction, after time $T_1$ a random point $\tau$ is accepted for $\underline{\eta}^+$ (respectively $\underline{\eta}^-$) if and only if  \[
{\exp[-\eta^+_i(\tau)\lambda_i^+(\tau)]\over e^{4\b}}>U_\tau, \:\:\left(\mbox{respectively } {\exp[-\eta^-_i(\tau)\lambda_i^-(\tau)]\over e^{4\b}}>U_\tau \right)
\]
where $U_{\tau}$ is a uniform random variable associated with $\tau$. \\
Notice that, under the event $A_N\cap B_N$ (see again the proof of Theorem \ref{th:covering}), it holds that
\begin{equation}
\label{localfield_bounds} {\exp[-\eta^-_i(t)\lambda_i^-(t)]}\leq {\exp[-\sigma_i(t)\lambda_i(t)]} \leq {\exp[-\eta^+_i(t)\lambda_i^+(t)]}
\end{equation}
for any $t$ up to $T_1+T_c$, which means that any point which is accepted for a spin flip for $\underline{\eta}^-$ is also accepted for $\underline{\hat{\sigma}}$, and any point which is accepted for a spin flip for $\underline{\hat{\sigma}}$ is also accepted for $\underline{\eta}^+$, therefore, since we constructed a \textit{monotone coupling} (see \cite{L13}), it holds
\begin{equation}
\label{spins_bounds} \eta_i^-(t)\leq \hat{\sigma}_i(t)\leq \eta_i^+(t), \:\:\:\:i\in\Lambda_N,
\end{equation}
for any $t$ up to $T_1+T_c$.
Actually, \eqref{localfield_bounds} and \eqref{spins_bounds} are true only up to the second to last flip of $\underline{\hat{\s}}$, since its last flip occurs with rate of order $e^{4\b+L_N}$: anyway this is not so important. In fact, if we denote by $m_N^+$ (respectively $m^-_N$) the total magnetization associated with $\underline{\eta}^+$ (respectively $\underline{\eta}^-$), our goal is to give bounds on $m_N$ by means of $m_N^+$ and $m_N^-$: it is true that, up to the second to last flip of $\underline{\hat{\s}}$, by \eqref{spins_bounds} it holds that
\[m_N^-(t)\leq m_N(t)\leq m_N^+(t).\]
To extend the bounds on the whole interval $[0,T_1+T_c]$ it is sufficient to add a term $+{2\over N}$ to the upper bound.
These bounds are true also passing to the time-scaled processes $\tilde{m}_N^{\pm}(t)={m}_N^{\pm}(\theta_N(t))$: to sum up, with this construction, under the event $A_N\cap B_N$, it holds
\begin{equation*}
 \tilde{m}_N^-(t)\leq \tilde{m}_N(t)\leq \tilde{m}_N^+(t)+{2\over N} \: \mbox{ for any }\: t\in[0,\tau_{\tilde{m}}]
\end{equation*}
which implies, again under $A_N\cap B_N$,
\begin{equation}
\label{magnetizations_bounds} \max\{||\tilde{m}_N-\tilde{m}_N^-||_{\infty}, ||\tilde{m}_N-(\tilde{m}_N^++{2\over N})||_{\infty}\}\leq ||\tilde{m}_N^+-\tilde{m}_N^-||_{\infty}+{2\over N},
\end{equation}
where $||\cdot||_\infty$ denotes the uniform norm on $\mathcal{D}([0,T\wedge \tau_{\tilde{m}}],\R)$.
Notice that by the graphical construction and the definition of $\tilde{m}_N^+$ and $\tilde{m}_N^-$, one gets that \[||\tilde{m}_N^+-\tilde{m}_N^-||_{\infty}\to 0 \:\:\mbox{ in probability as } \: \b,N\uparrow+\infty.\] 
 Hence, since $P(A_N\cap B_N)\to 1$, thanks to \eqref{magnetizations_bounds} we also obtain that
\begin{equation}
\label{ch_three_convinprobability}  ||\tilde{m}_N-\tilde{m}_N^-||_{\infty}\to 0 \:\:\mbox{ in $L^1$ as } \: \b,N\uparrow+\infty,
\end{equation}
where the convergence in $L^1$ follows by convergence in probability and uniform integrability, the latter due to the fact that $||\tilde{m}_N-\tilde{m}_N^-||_{\infty}\leq 2$ for any $N$.\\
Observe that $\{\tilde{m}_N^-\}_N$ (but also $\{\tilde{m}_N^+\}_N$), stopped as soon as it reaches $+1$, converges to the process $\tilde{x}(t\wedge \tau_{\tilde{x}})$. Indeed from the definition of time scaling \eqref{timescaling}, after time $T_1$,
 $\tilde{m}^{\pm}_N$ essentially become Poisson processes rescaled by ${2\over N}$ with random intensity $Ne^{-X_N}$. So the deterministic limit follows from standard scaling arguments. \\
Now denote with $\mu_N$ the law of $\tilde{m}_N$ on $\mathcal{D}([0,T],\mathbb{R})$. Let also $\mu_x$ be the law of the limiting process $\tilde{x}$. To show the weak convergence of $\tilde{m}_N$ to $\tilde{x}$ it is enough to show that $\mathcal{W}_1(\mu_N,\mu_x)$ converges to 0 (see \cite{Vil08}). Let $\mu_N^-$ be the law of the process $\tilde{m}_N^-$. Since $\mu_N$ and $\mu_N^-$ can be coupled via the graphical construction of $\tilde{m}_N$ and $\tilde{m}_N^-$, by \eqref{ch_three_convinprobability} and the definition of the Wasserstein distance, it holds
\[ \mathcal{W}_1(\mu_N,\mu_N^-)\leq E\left[ ||\tilde{m}_N-\tilde{m}_N^-||_{\infty}\right] \to 0 \:\: \mbox{ as } \: \b,N\uparrow+\infty.\]
Therefore, by the fact that $\tilde{m}_N^-$ weakly converges to $\tilde{x}$,
\[\mathcal{W}_1(\mu_N,\mu_x) \leq \mathcal{W}_1(\mu_N,\mu_N^-)+\mathcal{W}_1(\mu_N^-,\mu_x)\to 0\:\: \mbox{ as } \: \b,N\uparrow+\infty ,  \]
which proves the weak convergence of $\tilde{m}_N$ to $\tilde{x}$.

\subsection{Smoothly varying initial condition: Proof of Theorem \ref{th:varying}}

{\em Proof of part (i).} The proof begins as that of Theorem \ref{th:firstspinflip}. Formulas \eqref{dtf}-\eqref{coex} hold unchanged, as constance in the initial condition is not used. Formula \eqref{T1est} is now replaced by (where we write $T_1$ for $T_{1,1}$)
\begin{equation}
\label{ztf}
%\forall t\geq 0\qquad
P(T_1>t)\,=\,
	%P\big(\big|Y_{N,\beta}\big|>\ln\beta\big)\,+\,\cr
	\exp\Bigg(-
S_{N,\beta}(t)\,
	%{\displaystyle \exp\Big( Y_{N,\beta} e^{-\alpha t } \Big)}
		 \left(1+
	 o({1}) + O\left( \frac{e^{\a t}}{\b} \right)
	 \right)
	%\Big(1+ O\Big( \frac{ e^{\alpha t } }{\beta } \Big) \Big)
	%\Big(1+O\Big(\frac{ e^{\alpha t } }{\beta } \Big)\Big)
	 +
	O\Big(
	\frac{N e^{-\beta } }{\beta\alpha }\Big)
	 \Bigg)\, ,
\end{equation}
where
\begin{equation} \label{SNbeta}
S_{N,\beta}(t)\,=\,
	\sum_{1\leq i\leq N}
	\frac{\displaystyle\exp\Big( 
	%-4\beta e^{-\alpha t } 
	\lambda_{N,\beta}
	\Phi\Big(\frac{i}{N}\Big)
	 e^{-\alpha t } \Big)}
	{\Big(\displaystyle
	-\Phi\Big(\frac{i}{N}\Big)
	\Big)
	 \alpha 
\lambda_{N,\beta}
	 e^{-\alpha t } 
	}
	 \,.
\end{equation}
To alleviate the notation, 
we set
$$\gamma\,=\,
-\lambda_{N,\beta} 
e^{-\alpha t }\,,$$
so that
$S_{N,\beta}(t)$ can be rewritten as
$$S_{N,\gamma}(t)\,=\,
	\sum_{1\leq i\leq N}
	\frac{\displaystyle\exp\Big(
-\gamma\Phi\Big(\frac{i}{N}\Big)
	 \Big)}
	{\displaystyle
\gamma	 \alpha\Phi\Big(\frac{i}{N}\Big)
	}
	 \,.$$
 The technique is standard, indeed this expression looks like a Riemann sum, and if it were a genuine integral,
 we would apply directly Laplace's method of expansion. So we adapt here the technique to the discrete sum.
\begin{proposition}
	\label{expos}
In the regime where $N,\gamma$ tend to $+\infty$, we have
$$S_{N,\gamma}(t)\,\sim\,
	\frac{Ne^{
	-\gamma \Phi(x_*)
	}}
	{\displaystyle \gamma^{3/2}	 \alpha \Phi(x_*) }
	\sqrt{\frac{2\pi}{ 
	{\Phi''(x_*) }}}
	%\frac{1}{\alpha\Phi(x^*)}
	%\sqrt{\frac{2\pi}{-\gamma\Phi''(x^*)}}\exp\big({-\gamma\Phi(x^*)}\big)
	\,.$$
\end{proposition}
\begin{proof}
We employ the classical strategy devised by Laplace.
We expand the function $\Phi$ around its minimum and we approximate it from above by
a quadratic form.
Let $\varepsilon>0$. There exists
$\delta>0$ such that
\begin{multline} 
	\label{ino1}
	\forall x\in\,]x_*-\delta,x_*+\delta[\qquad\cr
	\Phi(x_*)+\frac{1-\varepsilon}{2}\Phi''(x_*)(x-x_*)^2\,\leq\,
	\Phi(x)\,\leq\,\Phi(x_*)+\frac{1+\varepsilon}{2}\Phi''(x_*)(x-x_*)^2\,.
\end{multline}
Since $x_*$ is the unique global minimum of $\Phi$, there exists $\eta>0$ such that
\begin{equation} 
	\label{ino2}
\forall x\in\,[0,1]\setminus ]x_*-\delta,x_*+\delta[\qquad
\Phi(x)\,\geq\,\Phi(x_*)+\eta\,.
\end{equation} 
With the help of the inequalities~\eqref{ino1} and~\eqref{ino2}, 
we split the sum and we get
\begin{multline}
	\label{fggh}
S_{N,\gamma}(t)\,=\,
	\sum_{i:|\frac{i}{N}-x_*|<\delta}
	\frac{\displaystyle\exp\Big(
-\gamma\Phi\Big(\frac{i}{N}\Big) \Big)}
	{\displaystyle 
	\gamma	 \alpha \Phi\Big(\frac{i}{N}\Big) 
	}
	\,+\,
	\sum_{i:|\frac{i}{N}-x_*|\geq\delta}
	\frac{\displaystyle\exp\Big(
-\gamma\Phi\Big(\frac{i}{N}\Big) \Big)}
	{\displaystyle 
	\gamma	 \alpha\Phi\Big(\frac{i}{N}\Big) 
	 }
	\cr
	\,\leq\,
	\sum_{i:|\frac{i}{N}-x_*|<\delta}
	\frac{\displaystyle\exp\Big(
	%{\exp\Big(
	-\gamma \Phi(x_*)-\gamma \frac{1-\varepsilon}{2}\Phi''(x_*)\big(\frac{i}{N}-x_*\big)^2
	\Big)}
	{\displaystyle \gamma	 \alpha 
	  \Phi(x_*) }
	%{\displaystyle \gamma	 \alpha \Phi\Big(\frac{i}{N}\Big) }
	\cr
	\,+\,
	\sum_{i:|\frac{i}{N}-x_*|\geq\delta}
	%\frac{\displaystyle\exp\Big( -\gamma\big(\Phi(x^*)+\eta\big)}
	\frac{\displaystyle e^{- \gamma(\Phi(x_*)+\eta)}}
	{\displaystyle \gamma	 \alpha 
	\Phi(x_*) }
	%{\displaystyle \gamma	 \alpha \Phi\Big(\frac{i}{N}\Big) }
	 \,.
\end{multline}
Let us focus on the first sum. Setting
$$\Delta\,=\,\gamma
	\frac{1-\varepsilon}{2}\Phi''(x_*)\,,$$
	we estimate the sum as follows:
\begin{multline*}
	\sum_{i:|\frac{i}{N}-x_*|<\delta}
	%\frac{\displaystyle\exp\Big(
	%{\exp\Big( \Delta\big(\frac{i}{N}-x^*\big)^2
	e^{- \Delta(\frac{i}{N}-x_*)^2}
	\,=\,
	\kern-11pt
	\sum_{i:x_*-\delta<\frac{i}{N}< x_*-\frac{1}{N} }
	\kern-11pt
	e^{- \Delta(\frac{i}{N}-x_*)^2}
	\kern-3pt
	+2+
	\kern-11pt
	\sum_{i:x_*+\frac{1}{N}\leq\frac{i}{N}<x_*+\delta}
	\kern-11pt
	e^{- \Delta(\frac{i}{N}-x_*)^2}
	\cr
	\,\leq\,
	\kern-11pt
	\sum_{i:x_*-\delta<\frac{i}{N}< x_*-\frac{1}{N} }
	\kern-11pt
	N\int_{\frac{i}{N}}^{\frac{i+1}{N}}
	e^{- \Delta(x-x_*)^2}\,dx
	\kern-0pt
	+2+
	\kern-11pt
	\sum_{i:x_*+\frac{1}{N}\leq\frac{i}{N}<x_*+\delta}
	\kern-11pt
	N\int_{\frac{i-1}{N}}^{\frac{i}{N}}
	e^{- \Delta(x-x_*)^2}\,dx
	\cr
	\,\leq\,
	N\int_{x_*-\delta}^{x_*+\delta}
	\kern-11pt
	e^{- \Delta(x-x_*)^2}\,dx+2
	\,\leq\,
	N\int_{-\infty}^{+\infty}
	\kern-11pt
	e^{- \Delta(x-x_*)^2}\,dx+2
	\,=\,
	N\sqrt{\frac{\pi}{{\Delta}}}+2\,.
	%\cdots
	%{\displaystyle \gamma	 \alpha \Phi\Big(\frac{i}{N}\Big) }
\end{multline*}
Reporting this inequality in~\eqref{fggh}, we obtain
\begin{equation*}
	%\label{fgg}
S_{N,\gamma}(t)\,\leq\,
	\frac{N e^{-
	\gamma \Phi(x_*)
	}}
	%{\displaystyle \gamma	 \alpha \big(\Phi(x^*)-\varepsilon\big) }
	{\displaystyle \gamma	 \alpha 
	 \Phi(x^*) }
	\Big(\sqrt{\frac{2\pi}{ {\gamma(1-\varepsilon)\Phi''(x^*) }}}+\frac{2}{N}+e^{-\gamma\eta}\Big)
	%{\displaystyle \gamma	 \alpha \phi\Big(\frac{i}{N}\Big) }
	%\,+\,
	%
	%\frac{\displaystyle\exp\Big( -\gamma\big(\phi(x^*)+\eta\big)}
	%frac{\displaystyle e^{ \gamma(\phi(x^*)+\eta)}}
	%\displaystyle \gamma	 \alpha 
	  %big(4/c-\phi(x^*)\big) }
	\,.
\end{equation*}
The last two terms in the parenthesis are negligible compared to the first, so that, for $N,\gamma$ large enough, we have
\begin{equation}
	\label{afgg}
S_{N,\gamma}(t)\,\leq\,
	\frac{Ne^{ -\gamma \Phi(x_*) }}
	{\displaystyle \gamma^{3/2}	 \alpha 
	 \Phi(x_*) }
	\sqrt{\frac{2\pi}{ { \Phi''(x_*) }}}
	\frac{1+\varepsilon}{1-\varepsilon}
	\,.
\end{equation}
We seek next 
a similar inequality in the opposite direction.
By
inequality~\eqref{ino1}, we have
\begin{multline}
	\label{fgggh}
S_{N,\gamma}(t)\,\geq\,
	\sum_{i:|\frac{i}{N}-x_*|<\delta}
	\frac{\displaystyle\exp\Big(
-\gamma\Phi\Big(\frac{i}{N}\Big) \Big)}
	{\displaystyle 
	\gamma	 \alpha \Phi\Big(\frac{i}{N} \Big)
	  }
	\cr
	\,\geq\,
	\sum_{i:|\frac{i}{N}-x_*|<\delta}
	\frac{\displaystyle\exp\Big(
	%{\exp\Big(
	-\gamma \Phi(x_*)-\gamma \frac{1+\varepsilon}{2}\Phi''(x_*)\big(\frac{i}{N}-x_*\big)^2
	\Big)}
	{\displaystyle 
	\gamma	 \alpha \Big(\Phi(x_*)+\frac{1+\varepsilon}{2}\Phi''(x_*)\delta^2\Big)
	}
	%{\displaystyle \gamma	 \alpha \Phi\Big(\frac{i}{N}\Big) }
	 \,.
\end{multline}
Setting
$$\Delta\,=\,\gamma
	\frac{1+\varepsilon}{2}\Phi''(x_*)\,,$$
	we estimate the sum as follows:
\begin{multline*}
	\sum_{i:|\frac{i}{N}-x_*|<\delta}
	%\frac{\displaystyle\exp\Big(
	%{\exp\Big( \Delta\big(\frac{i}{N}-x^*\big)^2
	e^{- \Delta(\frac{i}{N}-x_*)^2}
	\,\geq\,
	\kern-11pt
	\sum_{i:x_*-\delta<\frac{i}{N}< x_*-\frac{1}{N} }
	\kern-11pt
	e^{- \Delta(\frac{i}{N}-x_*)^2}
	+
	\kern-11pt
	\sum_{i:x_*+\frac{1}{N}\leq\frac{i}{N}<x_*+\delta}
	\kern-11pt
	e^{- \Delta(\frac{i}{N}-x_*)^2}
	\cr
	\,\geq\,
	\kern-11pt
	\sum_{i:x_*-\delta<\frac{i}{N}< x_*-\frac{1}{N} }
	\kern-11pt
	N\int_{\frac{i-1}{N}}^{\frac{i}{N}}
	e^{- \Delta(x-x_*)^2}\,dx
	+
	\kern-11pt
	\sum_{i:x_*+\frac{1}{N}\leq\frac{i}{N}<x_*+\delta}
	\kern-11pt
	N\int_{\frac{i}{N}}^{\frac{i+1}{N}}
	e^{- \Delta(x-x_*)^2}\,dx
	\cr
	\,\geq\,
	N\int_{x_*-\delta}^{x_*+\delta}
	e^{- \Delta(x-x_*)^2}\,dx-
	N\int_{x_*-2/N}^{x_*+2/N}
	e^{- \Delta(x-x_*)^2}\,dx
	\cr
	\,\geq\,
	\sqrt{\frac{N}{{\Delta}}}
	\int_{-\sqrt{\Delta}\delta}^{\sqrt{\Delta}\delta}
	e^{- x^2}\,dx-4
	\,\geq\,
	N\sqrt{\frac{\pi}{{\Delta}}}(1-\varepsilon)\,,
	%\cdots
	%{\displaystyle \gamma	 \alpha \Phi\Big(\frac{i}{N}\Big) }
\end{multline*}
where the last inequality holds 
for $N,\gamma$ large enough. Plugging this inequality into~\eqref{fgggh}, we get
\begin{equation}
	\label{bfgg}
S_{N,\gamma}(t)\,\geq\,
	\frac{-Ne^{ \gamma \Phi(x_*) }}
	{\displaystyle \gamma^{3/2}	 \alpha 
	\Big(\Phi(x_*)+\frac{1+\varepsilon}{2}\Phi''(x_*)\delta^2\Big)
	}
	\sqrt{\frac{2\pi}{ { \Phi''(x_*) }}}
	\frac{1-\varepsilon}{1+\varepsilon}
	\,.
\end{equation}
Inequalities~\eqref{afgg} and~\eqref{bfgg} yield
the asymptotic expansion stated in the proposition.
\end{proof}
Inserting the estimate obtained for $S_{N,\g}$ in \eqref{ztf} we obtain
\begin{equation} \label{ztf1}
P(T_1>t)\,=\,
	%P\big(\big|Y_{N,\beta}\big|>\ln\beta\big)\,+\,\cr
	\exp\Bigg(-
\frac{Ne^{
	\lambda_{N,\beta} 
e^{-\alpha t } \Phi(x_*)
	}}
	{\displaystyle \left(-\lambda_{N,\beta} 
e^{-\alpha t }\right)^{3/2}	 \alpha \Phi(x_*) }
	\sqrt{\frac{2\pi}{ 
	{\Phi''(x_*) }}}
\,
	%{\displaystyle \exp\Big( Y_{N,\beta} e^{-\alpha t } \Big)}
		 \left(1+
	 o({1}) + O\left( \frac{e^{\a t}}{\b} \right)
	 \right)
	%\Big(1+ O\Big( \frac{ e^{\alpha t } }{\beta } \Big) \Big)
	%\Big(1+O\Big(\frac{ e^{\alpha t } }{\beta } \Big)\Big)
	 +
	O\Big(
	\frac{N e^{-\beta } }{\beta\alpha }\Big)
	 \Bigg)\, .
\end{equation}
At this point we proceed as in the proof of Theorem \ref{th:firstspinflip}. Set
\[
t :=  \frac{1}{\a} \ln \left( - \frac{\l_{N,\b}}{\ln N} \right) 
+ \frac{3}{2\a} \frac{\ln \ln N}{\ln N} 
+ \frac{1}{\a \ln N} \ln \left( \a \sqrt{\frac{\Phi''(x_*)}{2 \pi}}\right) + \frac{v}{\a \ln N}  .
\]
Inserting in \eqref{ztf1} and recalling that $\Phi(x_*) = 1$ we get
\[
P(T_1>t)  = \exp\Big(-
	%{\displaystyle \exp\Big( Y_{N,\beta} e^{-\alpha t } \Big)}
	\displaystyle 
	\frac{1}{ \alpha }
	%\sqrt{\frac{2\pi\Phi(x^*)}{ { \Phi''(x^*) }}}
	\exp\big(v+
	%Y_{N,\beta} \Big(\frac{\ln N}{ -\lambda_{N,\beta} }\Big)
		  o({1}) \big)
	 \Big)\,,
\]
which completes the proof of part (i).
\bigskip

\noindent
{\em Proof of part (ii)}. This is identical to the proof of Theorem \ref{th:covering}. Some care is only needed when proving that undesired flips occur with small probability. For the proof given in Section 3.2 (Step 1) to go through with no changes one uses the assumption $\Phi(x) \in [1,2]$ for every $x$.
\bigskip

\noindent
{\em Proof of part (iii)}. Given the results in part (i), up to a global change of sign we need to find the asymptotic distribution of the first spin flip time $T_1$ starting with $\s_i(0) = -1$ and 
\begin{equation} \label{lambdainit1}
\l_i(0) = -4\b + \l_{N,\b} \Phi\left(\frac{i}{N}\right)  + o(i,N,\b)
\end{equation}
with
\[
\l_{N,\b} = \ln N - \frac32 \ln \ln N + O(1).
\]
The identity \eqref{ztf} becomes
\begin{equation}
\label{ztf2}
%\forall t\geq 0\qquad
P(T_1>t)\,=\,
	%P\big(\big|Y_{N,\beta}\big|>\ln\beta\big)\,+\,\cr
	\exp\Bigg(-
S_{N,\beta}(t)\,e^{-4\b e^{-\a t}}
	%{\displaystyle \exp\Big( Y_{N,\beta} e^{-\alpha t } \Big)}
		 \left(1+
	 o({1}) + O\left( \frac{e^{\a t}}{\b} \right)
	 \right)
	%\Big(1+ O\Big( \frac{ e^{\alpha t } }{\beta } \Big) \Big)
	%\Big(1+O\Big(\frac{ e^{\alpha t } }{\beta } \Big)\Big)
	 +
	O\Big(
	\frac{N e^{-\beta } }{\beta\alpha }\Big)
	 \Bigg)\, ,
\end{equation}
and now
\begin{equation} \label{SNbetabis}
S_{N,\beta}(t)\,=\,
	\sum_{1\leq i\leq N}
	\frac{\displaystyle\exp\Big( 
	%-4\beta e^{-\alpha t } 
	\lambda_{N,\beta}
	\Phi\Big(\frac{i}{N}\Big)
	 e^{-\alpha t } \Big)}
	{\Big(\displaystyle
	\frac{4}{c}-\Phi\Big(\frac{i}{N}\Big)
	\Big)
	 \alpha 
\lambda_{N,\beta}
	 e^{-\alpha t } 
	}
	 \,,
\end{equation}
where we have used the fact that 
\[
4 \b - \l_{N,\b} \Phi\left(\frac{i}{N}\right) \sim \l_{N,\b}\left[\frac{4}{c} - \Phi\Big(\frac{i}{N}\Big)\right].
\]
The asymptotics of $S_{N,\b}$ are obtained as in the proof of part (i), and we obtain
\[
S_{N,\b} = \frac{Ne^{
	\lambda_{N,\beta} e^{-\alpha t }
	 \Phi(x^*)
	}}
	{\displaystyle 
	(\lambda_{N,\beta} e^{-\alpha t })^{3/2}
	\alpha \big(4/c-\Phi(x^*)\big) }
	\sqrt{\frac{2\pi}{ 
	{- \Phi''(x^*) }}}\big(1+o(1)\big)
	\]
giving
\[
P(T_1>t)\,=\,
	%P\big(\big|Y_{N,\beta}\big|>\ln\beta\big)\,+\,\cr
	\exp\Bigg(-
	\frac{Ne^{
	\lambda_{N,\beta} e^{-\alpha t }
	 \Phi(x^*)
	}}
	{\displaystyle 
	(\lambda_{N,\beta} e^{-\alpha t })^{3/2}
	\alpha \big(4/c-\Phi(x^*)\big) }
	\sqrt{\frac{2\pi}{ 
	{- \Phi''(x^*) }}}
	%{\displaystyle \exp\Big( Y_{N,\beta} e^{-\alpha t } \Big)}
	{\displaystyle e^{
	-4\beta 
		 e^{-\alpha t } }}
		 \big(1+
	 o_P({1})\big)
	 \Bigg)\,.
\]
Choosing
\[
t := \frac{1}{\alpha}\ln \Big(\frac{4}{c}-\Phi(x^*)\Big) 
+ \frac{3}{2\a} \frac{\ln \ln N}{\ln N} 
+ \frac{1}{\a \ln N} \ln \left( \a \sqrt{\frac{-\Phi''(x^*)}{2 \pi}}\right) + \frac{v}{\a \ln N}  .
\]
we obtain
\[
P(T_1>t)  = \exp\Big(-
	%{\displaystyle \exp\Big( Y_{N,\beta} e^{-\alpha t } \Big)}
	\displaystyle 
	\frac{1}{ \alpha }
	%\sqrt{\frac{2\pi\phi(x^*)}{ { \phi''(x^*) }}}
	\exp\big(v+
	%Y_{N,\beta} \Big(\frac{\ln N}{ -\lambda_{N,\beta} }\Big)
		  o({1}) \big)
	 \Big)\,.
\]
It follows that 
\[
T_1 = \frac{1}{\alpha}\ln \Big(\frac{4}{c}-\Phi(x^*)\Big) 
+ \frac{3}{2\a} \frac{\ln \ln N}{\ln N} 
+ \frac{1}{\a \ln N} \ln \left( \a \sqrt{\frac{-\Phi''(x*)}{2 \pi}}\right) + \frac{X_N}{\ln N}
\]
where $X_N$ converges in distribution to a random variable $X$ whose distribution is given by
\[
P(X>x) = \exp\left(- e^x \right).
\]
Therefore
\begin{multline}
	\label{gnh}
\lambda_i(T_1^-)\,=\,
	%\lambda_{N,\beta}\,
	\lambda_i(0)\,
 e^{-\alpha T_1}\cr
\,=\,
 %e^{-\alpha T_1}
	\frac{
	\Big(-4\beta +\lambda_{N,\beta}
	\Phi\Big(\frac{i}{N}\Big)+ o(1)\Big)}{\displaystyle\frac{4}{c}-\Phi(x^*)} \left(
	1-
	\frac{3}{2\alpha}\frac{\ln\ln N}{\ln N}
	+ \frac{1}{\a \ln N} \ln \left( \a \sqrt{\frac{-\Phi''(x^*)}{2 \pi}}\right) + \frac{X_N}{\ln N} + o\left(\frac{1}{\ln N}\right)
	\right)\cr
\,=\,
	\frac{\displaystyle-\frac{4}{c}+ \Phi\Big(\frac{i}{N}\Big)}{\displaystyle\frac{4}{c}-\Phi(x^*)} 
	\Big(
	\ln N-
	\frac{3}{2}{\ln\ln N}
	+X_N
	\Big)
	+ o(1)
	\,.
\end{multline}
After the droplet expansion
\[
\lambda_i(T_1+T_c)\,=\, \l_i(T_1^-) + 4 \b = 4 \b  -\frac{\displaystyle \frac{4}{c}- \Phi\Big(\frac{i}{N}\Big)}{\displaystyle\frac{4}{c}-\Phi(x^*)} 
	\Big(
	\ln N-
	\frac{3}{2}{\ln\ln N}
	+X_N
	\Big)
	+ o(1)
\]
which is of the same form as \eqref{lambdainit1} with $\Phi$ replaced by $R\Phi$. The proof of part (iii) thus follows by iterating this argument.

\bibliographystyle{abbrv}

\end{document}